\documentclass[12pt,twoside,english,reqno,a4paper,final]{amsart}
\usepackage[utf8]{inputenc}
\usepackage[english]{babel}

\usepackage[lined,boxed]{algorithm2e}
\usepackage{float} %needed to fix the position of  a figure
\usepackage[normalem]{ulem}
\usepackage{hyperref}
\usepackage{fullpage}

\usepackage{listings,graphicx,amsmath,varioref,amscd,amssymb,color,xcolor,bm,amsthm,amsfonts,graphics,lineno,float,comment}
\usepackage{mathrsfs} % mathscr F

\usepackage{wrapfig}
\newcommand{\ch}[1]{#1}

\theoremstyle{plain}
\newtheorem{defin}{Definition}[section]
\newtheorem{theorem}[defin]{Theorem}
\newtheorem{prop}[defin]{Proposition}
\newtheorem{lemma}[defin]{Lemma}

\newtheorem{coro}[defin]{Corollary}
\newtheorem{remark}[defin]{Remark}

\numberwithin{equation}{section}

\renewcommand{\theequation}{\thesection.\arabic{equation}}

\DeclareMathOperator{\R}{\mathbb{R}}
\DeclareMathOperator{\N}{\mathbb{N}}

\newcommand{\car}[1]{\raise1pt\hbox{$\chi$}_{#1}}

\definecolor{sap}{RGB}{120,36,51}

\newcommand{\eps}{\varepsilon}

\def\lp'n{(L^{p'}(\Omega))^{N}}

\def\R{\mathbb{R}}
\def\N{\mathbb{N}}

\def\car#1{\chi_{_{{#1}}}}
\def\norma#1#2{\|#1\|_{\lower 4pt \hbox{$ \scriptstyle #2$ }}}

\newcommand{\tbar}{\overline{\bm\vartheta}}
\newcommand{\vecth}{{\bm\vartheta}}
\newcommand{\vth}{\vartheta}
\newcommand{\vthb}{\bar{\vartheta}}
\def\hl{H^{1}_{0L}(I)}
\newcommand{\vh}{\bm{\vec{h}}}

\newcommand{\vht}{\bm{\vec{\tilde h}}}
\newcommand{\vhb}{\bm{\vec{\bar h}}}
\newcommand{\io}{\int_0^1}
\def\1eq{$(P_{\vartheta_i})$}
\def\2eq{$(P_{\lambda_i})$}
\def\3eq{$(P_{\alpha})$}
\newcommand{\aut}{\mu}

\newcommand{\HH}{\max\limits_{i=1,\dots,n}\{|\vec h_i|\}}

\author[1]{Riccardo Durastanti}

\address[1]{Dipartimento di Scienze di Base e Applicate per l' Ingegneria, ``Sapienza" Universit\`a di Roma, Via Scarpa 16, 00161 Roma, Italy
\\ riccardo.durastanti@sbai.uniroma1.it}

\author[2]{Lorenzo Giacomelli}

\address[2]{Dipartimento di Scienze di Base e Applicate per l' Ingegneria, ``Sapienza" Universit\`a di Roma, Via Scarpa 16, 00161 Roma, Italy
\\ lorenzo.giacomelli@sbai.uniroma1.it}

\author[3]{Giuseppe Tomassetti}

\address[3]{Universit\`a degli Studi ``Roma Tre'', Dipartimento di Ingegneria, Via Vito Volterra 62, 00146 Roma, Italy
  \\ giuseppe.tomassetti@uniroma3.it\\
\href{https://orcid.org/0000-0001-8801-7461}{\texttt{ORCiD: 0000-0001-8801-7461}}
}


\begin{document}

\title{Shape programming of a magnetic elastica}

\maketitle

\begin{abstract}
We consider a cantilever beam which possesses a possibly  non-uniform permanent magnetization, and whose shape is controlled by an applied magnetic field. We model the beam as a plane elastic curve and we suppose that the magnetic field acts upon the beam by means of a distributed couple that pulls the magnetization towards its direction. Given a list of target shapes, we look for a design of the magnetization profile and for a list of controls such that the shapes assumed by the beam when acted upon by the controls are as close as possible to the targets, in an averaged sense. To this effect, we formulate and solve an  optimal design and control problem leading to the minimization of a functional which we study by both direct and indirect methods. In particular, we prove \ch{that minimizers exist, solve the associated Lagrange-multiplier formulation (besides non-generic cases), and are unique at least} for sufficiently low intensities of the controlling magnetic fields. To \ch{achieve the latter result,} we use two nested fixed-point arguments relying on the Lagrange-multiplier formulation of the problem, a method which also suggests a numerical scheme. \ch{Various relevant open question are also discussed.}
\vskip 0.5\baselineskip
\end{abstract}

\tableofcontents

\section{Introduction and main results}\label{sec:intro}

\subsection{Motivation}
Recent technological developments have made it possible to assemble\ch{, with pinpoint accuracy of composition and texture, elastic} materials which can convert into deformation, and hence motion, a diversity of energetic inputs in the form of heat, light, chemical agents, electric and magnetic fields. \ch{These advances make it possible to craft devices which can change their shape through distributed actuation, mimicking biological examples such as elephant trunks and octopus arms, which are best suited for interacting with complex environments \cite{mazzolai2012}. In particular, engineers find these materials appealing for applications at small scales, such as for instance microsurgery and drug delivery \cite{nelson2010}, where the implementation of conventional technologies proves difficult. For these applications, a key requirement is the ability to attain a large variety of shapes: for example, the locomotion of miniature robots based on crawling involves negotiation of obstacles of all sorts in confined spaces \cite{Hu2018}, thus requiring high adaptability; likewise, locomotion based on swimming  requires ad hoc shape-control strategies such as a distinct power- and recovery-stroke \cite{lum2016shape}.

\smallskip

Shape control of devices with distributed actuation cannot be addressed with the conventional engineering practice of designing separately power, kinematics, and control: in order to achieve a desired motion strategy, the device morphology and the stimulus must be designed \emph{at the same time} \cite{krichmar2012,pfeifer2014}. This state of matters} has stimulated substantial theoretical work concerning \emph{shape programming}, i.e., the design of textures and controls that produce desired shapes, a topic which is becoming increasingly relevant in \ch{theoretical} elasticity (see e.g. \cite{Acharya2018,Agostiniani2017a}).

\smallskip

\ch{Additional problems arise when miniature devices are required to operate untethered. In fact, since dissipative effects are dominant at small scales, self-powered devices require high-density energy storage mechanisms. In this respect, magnetic actuation offers several advantages: it can remotely provide both control and power}, it offers fast response, and it does not affect the surrounding medium by polarization \cite{Rikken2014a}.

\smallskip

\ch{Among the many available magneto-elastic materials, the so-called magnetorheological elastomers (MREs) are particular suited for shape programming. Originally devised as viscoelastic solids whose mechanical response could be controlled by an applied magnetic field \cite{Ginder1999}, MREs are obtained by embedding magnetic particles in a soft elastomeric matrix. Thanks to their compliance, MREs find applications in circumstances when large displacements are in need \cite{zhao1}. Moreover, their magnetic properties can be finely controlled \cite{kim2018}. Furthermore, the theoretical modeling of MREs is well established (see for instance \cite{Dorfmann2004b} and \cite{KANKANALA2004}), their stability at both the macroscopic \cite{Ottenio2008} and the microstructural level \cite{Rudykh2013} has been investigated, and ad-hoc computational techniques \cite{Pelteret2016} are available.}

\smallskip

\ch{Proofs of concept exist \cite{Hu2018,lum2016shape,Xu2019} that MREs can be used to fabricate small-scale untethered microrobots, which can walk, crawl, and swim. Indeed, w}hen crafted in the form of thin bodies, such as rods or plates, magnetorheological elastomers display a very large range of motion \cite{zhao1,zhao2}. In this respect, shape programming appears to be rather intriguing even for a simple mechanical model such Euler's Elastica\ch{, which is at the basis of the model we adopt in the present paper}. This is not surprising, since the qualitative and quantitative properties of equilibrium solutions for elastic curves in a diversity of settings is still the object of intense mathematical research (see e.g. \cite{DellaCorte2017,Ferone2016,jian2003,miura2017elastic}).

\smallskip

\ch{The present paper is meant as a contribution towards the development of a systematic mathematical framework for shape programming of magnetic materials, with an emphasis on obtaining rigorous results. With this aim in mind, we focus on a mechanical model featuring a planar \emph{cantilever beam} with \emph{permanent magnetization} having constant intensity but variable direction, under a \emph{spatially-constant magnetic field}, as shown in the following figure.}

\begin{figure}[H]\label{fig:1}
  \begin{center}
   \includegraphics[scale=0.8]{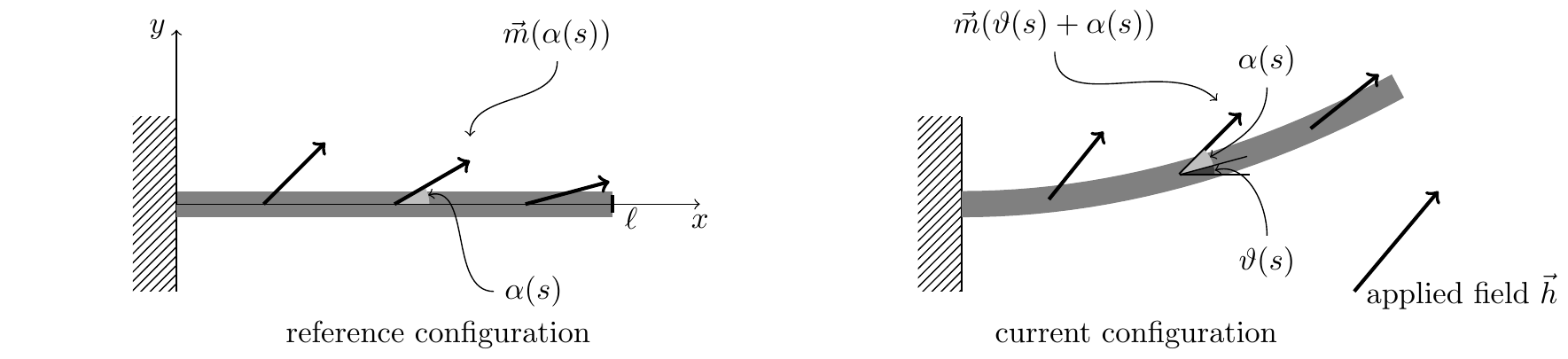}
  \end{center}
  \caption{\ch{A} cantilever beam with a permanent magnetization of uniform intensity and angle $\alpha(s)$ with respect to the tangent.}
\end{figure}
\ch{For this model, we formulate an optimal design-and-control problem which may verbally be described as follows: given a list of pre-assigned target shapes, choose a magnetization profile (the \emph{morphology}) and a list of applied fields (the \emph{stimuli}) such that the shapes attained by the beam under the action of these fields best approximate the given shapes, in some averaged sense. Admittedly, the formulation we choose ignores a certain amount of the physics which comes into play in actual engineering applications. For example, the interaction force with the surrounding medium is being ignored. Likewise, this formulation ignores the dynamic effects of moving from one shape to another as the applied field varies (see the discussion in Section \ref{sec:open}). On the other hand, experimental evidence from \cite{lum2016shape} shows that the results obtained in such simplified setting can still furnish valuable guidance to the design of actual devices.}

\subsection{\ch{The mathematical model}}
We model the cantilever beam as a magnetized {\emph{planar elastica}}, and we describe its configuration through the parametric curve $\ch{\vec r}:(0,1)\to\mathbb R^2$ defined by
\begin{equation}\label{eq:5}
  \ch{\vec r}(s)=\ell\int_0^s \vec m(\vartheta(\xi))\,{\rm d}\xi,
\end{equation}
where $\ell$ is the length of the beam, $\vec m:\mathbb R\to\mathbb R^2$ is defined by
  \begin{equation}\label{eq:17}
    \vec m(v)=(\cos(v),\sin(v)),
\end{equation}
and $\vartheta(s)$ is the \emph{rotation at $s$}. With slight abuse of language, we shall refer to the function $\vartheta:(0,1)\to\mathbb R$ as the \emph{shape of the beam}. Since the beam is clamped, the shape must satisfy the essential boundary condition:
\begin{equation*}
  \vartheta(0)=0,
\end{equation*}
which holds irrespectively of the loading environment.

\smallskip

The beam has a permanent magnetization per unit length, whose intensity is a constant $M_0$ (its unit in the S.I. System is ampere$\cdot$meter${}^{-2}$ $[\rm Am^{-2}]$), and whose orientation with respect to the tangent line is given by a possibly non-uniform \emph{relative angle} $\alpha(s)\in\mathbb R$, $s\in(0,1)$. We assume that the relative angle $\alpha(s)$ is not affected by the magnetic field and by the deformation process. Hence the vector field\ch{s}
\begin{equation*}
\ch{ \vec m(\alpha)=(\cos(\alpha),\sin(\alpha)),\quad }  \vec m(\vartheta+\alpha)=(\cos(\alpha+\vartheta),\sin(\alpha+\vartheta))
\end{equation*}
\ch{are} the \emph{orientation of the magnetization} in the \ch{undeformed, resp. deformed,} configuration\ch{s}.

\smallskip

\ch{Theories of magnetoelastic rods (see for instance \cite{cebers2007magnetic,Ciambella2018,gerbal2015}) predict} that when a spatially constant magnetic field $\vec H$ $[\rm Am^{-1}]$ is applied to the beam, any stable equilibrium configuration must be a local minimizer of the renormalized \emph{magnetoelastic energy}
\begin{equation}\label{eq:6}
  \mathcal E(\vartheta)=\int_0^1 \Big(\frac 12\big(\vartheta'\big)^2-\vec h\cdot\vec m(\vartheta+\alpha)\Big)\, {\rm d}s,
\end{equation}
where a dot denotes the scalar product, $\vec h$ is the renormalized magnetic field defined by $\vec h=\mu_0\frac {M_0 \ell^2 }{S}\vec H$, with
$\mu_0$ $[{\rm Hm^{-1}}]$ the magnetic permeability of vacuum  and  $S$ $[\rm Nm^2]$ is the \emph{bending stiffness}. The vector $\vec h$ is dimensionless, since its modulus  $|\vec h|=(\mu_0 M_0H\ell)/(S\ell^{-1})$ can be written as the ratio between the \emph{magnetic energy} $\frac 12\mu_0 M_0 H\ell$ that must be expended to immerse the beam in the magnetic field, and the \emph{elastic energy} $S/\ell$ that must be stored in the system to impart the curvature $\ell^{-1}$ to the beam.

\subsection{The state equation}\label{sec:the-state-equation}
For $\vec h=0$ the magnetoelastic energy has the unique minimizer $\vartheta=0$, which corresponds through \eqref{eq:5} to the straight configuration. As detailed in Section \ref{sec:basic} (see Corollary \ref{cor:theta}), for given, arbitrary $\vec h$ and $\alpha$ the magnetoelastic energy has at least one minimizer, which furthermore solves the \emph{Euler-Lagrange system}
\begin{equation}
\begin{cases}
\label{pthi1}\tag{$P_\vth$}
-\vth''-\vec h\cdot {D\vec m}(\alpha+\vartheta)=0   & \mbox{in $\ch{I:=}(0,1),$} \\
\vth(0)=0, \\
\vth'(1)=0,
\end{cases}
\end{equation}
where
\begin{equation}\label{consist}
  D\vec m(v)=(-\sin v,\cos v),\qquad \text{for all }v\in\mathbb R
\end{equation}
is the derivative of the function $\vec m$ defined in \eqref{eq:17}; moreover, such minimizer is unique if
\begin{equation}\label{eq:8}
  |\vec h|< c_p^{-2},
\end{equation}
where $c_p= 2/\pi$ is the best constant in the Poincar\'e-type inequality
\begin{equation*}
	\int_0^1 v^2\le c_p^2\int_0^1 (v')^2\quad\text{for all }v\in C^1([0,1])\text{ such that }v(0)=0.
\end{equation*}
The \emph{state equation} \eqref{pthi1} is a variant of the well-known \ch{{\emph{elastica equation}}}. Given $\alpha$, \eqref{pthi1} defines a \emph{solution operator}
\begin{equation}\label{def:Theta_alpha}
\Theta_\alpha\ : \quad B(0,c_p^{-2})\ni \vec h \ \longmapsto\ \vartheta= \Theta_\alpha(\vec h)
\end{equation}
which maps the \emph{control} $\vec h$ into the \emph{state} $\vartheta=\Theta_\alpha(\vec h)$. The manifold of attainable configurations parametrized by the chart $\Theta_\alpha$ is two-dimensional. Thus, one may hope that complex motions, such as for instance those required for applications to microswimmers \cite{alouges2015can,Alouges2013}, could be realized, at least with a reasonable approximation, by a judicious choice of a fixed magnetization profile and a time varying magnetic field. The papers \cite{Hu2018} and \cite{lum2016shape}  offer experimental evidence of this possibility.

\subsection{The optimal design-control problem}\label{sec:the-optimal-design}
In this paper we are concerned with the following situation. We are given a list of $n$ prescribed \emph{target shapes},
\[
  \tbar=(\vthb_1,\dots,\vthb_n):[0,1]\to\mathbb R^n,
\]
which the beam should ideally attain by applying $n$ \emph{controls}: these are the $n$ magnetic fields
\[
\vh=(\vec{h}_1,\dots,\vec{h}_n)\in\R^{2n},
\]
with $\vec{h}_i=(h_{ix},h_{iy})\in\R^2$. At our disposal is also a \emph{design}, the magnetization $\alpha$ of the beam. Thus, we look for a design $\alpha$ and a control $\vh$ such that the shapes $\vartheta_i = \Theta_\alpha(\vec h_i)$ attained by the beam when applying the magnetic fields $\vec h_i$, namely the solutions of
\begin{equation}
\begin{cases}
\label{pthi}\tag{$P_{\vth_i}$}
-\vth_i''-\vec h_i\cdot{D\vec m}(\alpha+\vth_i)=0  & \mbox{in $(0,1),$} \\
\vth_i(0)=0, \\
\vth_i'(1)=0,
\end{cases} \quad \mbox{$i=1,\dots,n,$}
\end{equation}
are ``as close as possible'' to the targets $\overline\vartheta_i$. The precise meaning of ``closeness'' depends on the choice of the {\emph cost functional} $\mathcal C$, which we define as follows:
\begin{equation}
\label{cost}
\mathcal{C}(\vh,\alpha,\vecth)\ch{=\mathcal C_{\eps,\gamma}(\vh,\alpha,\vecth)}=\frac{1}{2}\sum_{i=1}^n \io |\vth_i-\vthb_i|^2 +\frac{\varepsilon}{2}\io |\alpha'|^2 + \frac{\gamma}{2} \sum_{i=1}^n |\vec{h}_i|^2\ch{,}
\end{equation}
where $\varepsilon>0$ and $\gamma>0$ are positive parameters.

\begin{remark}[The cost functional]
{\rm The choice of the cost functional $\mathcal C$ in \eqref{cost} deserves a discussion. The first integral has an obvious interpretation, since we aim at minimizing the distance between the $n$ attained shapes $\vartheta_i$ and the $n$ target shapes $\overline\vartheta_i$. \ch{The second and third term, which penalize inhomogeneities of the magnetization density, resp. high intensities of the applied magnetic fields, are key technical ingredients, since they render the cost functional coercive with respect to a topology that guarantees compactness of minimizing sequences.}
}\end{remark}

Our precise mathematical formulation of the problem thus involves three ingredients:
\begin{itemize}
\item[(i)] the \emph{admissible space}
\begin{equation}\label{adm-H}
  \mathcal{H}=\{(\vh,\alpha,\vecth):\vh\in\R^{2n},\alpha\in\hl,\vecth\in\hl^n\}\color{black}=\R^{2n}\times\hl\times\hl^n,
\end{equation}
where
\begin{equation}\label{defhl}
  \hl:=\{v\in H^1(I):v(0)=0\};
\end{equation}
\item[(ii)] the \emph{cost functional} $\mathcal{C}:\mathcal{H}\to\R$, defined by \eqref{cost} for all $(\vh,\alpha,\vecth)\in\mathcal{H}$;
\smallskip
\item[(iii)] the \emph{admissible set}
\begin{equation}\label{eq:11}
\mathcal{A}=\left\{(\vh,\alpha,\vecth)\in \mathcal{H} : \ \vth_i \text{ solves \eqref{pthi} for every } i=1,\dots,n \right\}.
\end{equation}
\end{itemize}
With these three ingredients, we may formulate the following \emph{Optimal Control-Design Problem}:
\begin{equation}\label{main-problem}
\text{ minimize $\mathcal C(\vh,\alpha,\vecth)$ among all $(\vh,\alpha,\vecth)\in \mathcal A$.}
\end{equation}
\ch{Simple calculations using angle sum identities show that \eqref{pthi} and \eqref{cost} are invariant under an equal rotation of the vectors $\vec h_i$ and $\vec m(\alpha)$ (see also the proof of part $(iii)$ of Theorem \ref{min}): therefore, in \eqref{adm-H} we have set $\alpha(0)=0$ without losing generality.}

\subsection{Our results}

Using the direct method of the Calculus of Variations, we prove in Section \ref{sec:minimizer} the existence of a minimizer\ch{. In this respect, the penalization in the definition of $\mathcal C$ is crucial in guaranteeing coercivity for generic targets}:

\begin{theorem}
\label{min}
\begin{itemize}

\item[$(i)$] For any  $\eps>0$, $\gamma>0$, and any $\overline{\vecth}\in L^2(I)^n$, the cost functional $\mathcal C_{\ch{\eps,\gamma}}$ has a minimizer in the admissible set $\mathcal A$. Furthermore, any minimizer is such that
\begin{equation}\label{eq:13}
\left(\HH\right)^2\leq \frac{\overline{\Theta}^2}{\gamma},\qquad\mbox{where}\quad \overline{\Theta}^2= \sum_{i=1}^n \int_0^1\overline \vartheta_i^2.
\end{equation}
\ch{
\item[$(ii)$] For any {\em attainable} target $\tbar$, i.e. any $\tbar$ such that $(\vhb,\bar\alpha,\tbar)\in \mathcal A$ for some $\vhb\in\R^{2n}$  and some $\bar \alpha\in H^1_{0L}(I)$, minimizers of $\mathcal C_{\eps,\eps}$ converge to a minimizer of $\mathcal C_{0,0}$ as $\eps$ tends to $0$.

\item[$(iii)$] For $n=1$, any $\bar\vartheta \in H^3(I)$ with $\bar\vartheta(0)=0$ and $\bar\vartheta'(1)=0$ is attainable.
}
\end{itemize}
\end{theorem}

\ch{
\begin{remark}\label{epsgam}{\rm
For attainable targets, one has $\mathcal C_{0,0}(\vhb,\bar\alpha,\tbar)=0$, hence existence of minimizers of $\mathcal C_{0,0}$ is trivial. However, the class of attainable targets is a non-dense subset of $H^1_{0L}(I)^n$, and existence of minimizers of $\mathcal C_{0,0}$ seems to be nontrivial for generic targets (it is not even clear if $\inf \mathcal C_{0,0}$ will be positive or not). This motivates introducing the penalization terms, and part $(ii)$-$(iii)$ of Theorem \ref{min} legitimates this choice. We note on passing that any attainable $\tbar$ belongs to $H^3(I)^n$ (see Corollary \ref{cor:theta}), hence the assumption $\bar\vartheta\in H^3(I)$ in $(iii)$ is not restrictive.
}\end{remark}
}

An important consequence of \eqref{eq:13} is that, for $\overline\Theta$ sufficiently small and/or $\gamma$ sufficiently large, each of the applied magnetic field $\vec h_i$ satisfies \eqref{eq:8}; thus, if $(\bm{\vec h},\alpha,\bm\vth)$ is a minimizer with $\bm\vth=(\vth_1,\ldots,\vth_n)$, then each state $\vartheta_i$ is the unique solution of its state system \eqref{pthi}: this means that the mechanical equilibria identified by the minimization of $\mathcal C$ are stable. In other words, for $\overline\Theta$ sufficiently small and/or $\gamma$ sufficiently large each configuration $\vartheta_i$ corresponds to a stable minimizer of the magnetoelastic energy if the corresponding $\bm{\vec h}$ and $\alpha$ are taken as fixed.

\medskip

The previous result neither implies uniqueness of the {\em triplet} $(\bm{\vec h},\alpha,\bm\vth)$, nor provides a constructive scheme for its numerical approximation. Focusing on these two aspects, we investigate the Lagrange-multiplier reformulation of \eqref{main-problem}. This reformulation amounts to finding a critical point of the \emph{Lagrangian}
\[
\mathcal L (\vh,\alpha,\vecth,\bm{\lambda}):= \mathcal C (\vh,\alpha,\vecth) - \sum_{i=1}^n \int_0^1 \lambda_i \left(-\vth_i''-\vec h_i\cdot{D\vec m}(\alpha+\vth_i)\right),
\]
where $\bm{\lambda}=(\lambda_1,\dots,\lambda_n)$ is the \emph{Lagrange multiplier}. Differentiation of $\mathcal L$ yields, formally, the following system:
\begin{equation}
\label{EQ}
\left\{\begin{array}{lll}
(P_{\vth_i}): & \displaystyle -\vth_i''-\vec h_i\cdot{D\vec m}(\alpha+\vth_i)=0, & \vartheta_i(0)=\vartheta'_i(1)=0 \\
(P_{\lambda_i}): & \displaystyle -\lambda_i''-\lambda_i \vec h_i\cdot{D^2\vec m}(\alpha+\vth_i)=\vth_i-\vthb_i, & \lambda_i(0)=\lambda_i'(1)=0 \\
(P_{\alpha}): & \displaystyle -\varepsilon\alpha''+\sum_{i=1}^n \lambda_i\vec h_i\cdot{D^2\vec m}(\alpha+\vth_i)=0, & \ch{\eps}\alpha(0)=\ch{\eps}\alpha'(1)=0 \\
(P_{\vec h_i}): & \displaystyle \ch{\gamma}\vec h_{i}=-\io \lambda_i {D\vec m}(\alpha+\vth_i) &\\
\end{array}\right.
\end{equation}
for every $i=1,\dots,n$, where
\begin{equation}\label{consist2}
 D^2\vec m(v)=(-\cos v,-\sin v),\qquad v\in\mathbb R,
 \end{equation}
is the second derivative of $\vec m$. According to the standard theory of constrained minimization through Lagrange multipliers in Banach spaces, whose main results we summarize in the Appendix, a mimimizer  $(\bm{\vec h},\alpha,\bm\vartheta)$ of the cost functional in the admissible set corresponds to a stationary point $(\bm{\vec h},\alpha,\bm\vartheta,\bm\lambda)$ \ch{of $\mathcal L$} for some Lagrange multiplier $\bm\lambda$ only if that point is \emph{regular}, in the sense that the \emph{constraint mapping} $G:\mathcal H\to (\hl^n)'$ (the dual of $\hl^n$), defined by
\begin{equation}
\label{defg-intro}
\langle G(\vh,\alpha,\vecth),\bm{u}\rangle=\sum_{i=1}^n\left\{\io \vth_i'u_i'-\io \vec h_i\cdot{D\vec m}(\alpha+\vth_i) u_i\right\}\quad\mbox{for all $\bm{u}\in \hl^n$,}
\end{equation}
is \emph{Fr\'echet differentiable} at $(\bm{\vec h},\alpha,\bm\vartheta)$ and its differential $DG$ is \emph{surjective}. We apply this theory in Section \ch{\ref{sec:lagr2}}, where we study the Fr\'echet differentiability of the cost function $\mathcal C$ and of the constraint mapping $G$, as well as the surjectivity of the Fr\'echet differential of the latter. \ch{Let $r^+=\max\{r,0\}$, $r^-=\max\{-r,0\}\ge 0$ (so that $r=r^+-r^-$). We show:
\begin{theorem}\label{lagmul2}
Let $\eps\ge 0$, $\gamma\ge 0$, and let $(\vh,\alpha,\vecth)$ be a minimizer of $\mathcal{C}$ in $\mathcal{A}$.
\begin{itemize}
\item[$(i)$] if, for all $i=1,\ldots,n$, $\aut =1$ is not an eigenvalue of the Sturm-Liouville operator
\begin{equation}
\label{pippo}
\left\{
\begin{array}{ll}
- u'' + (r_i^-+ 1) u = \aut (r_i^+ + 1)u &\mbox{in $(0,1),$} \qquad r_i:= \vec h_i\cdot{D^2\vec m}(\alpha+\vth_i),
\\[-1ex] u(0)=0,\ u'(1)=0, &
\end{array}\right.
\end{equation}
then $(\vh,\alpha,\vecth)$ is a regular point of $\mathcal A$;
\item[$(ii)$] in particular, $(\vh,\alpha,\vecth)$ is a regular point of $\mathcal A$ if
\begin{equation}\label{bound-on-H}
  \HH<c_p^{-2};
  \end{equation}
  \item[$(iii)$]
if $(\vh,\alpha,\vecth)$ is a regular point of $\mathcal A$, then there exists a Lagrange multiplier $\bm{\lambda}\in \hl^n$ such that $(\vh,\alpha,\vecth,\bm{\lambda})$ is a solution of system \eqref{EQ}. Furthermore, $\alpha'(0)=0$ if $\eps>0$.
\end{itemize}
\end{theorem}

\begin{remark}{\rm
The condition in $(i)$ is equivalent to asking that the problem
\begin{equation}
\label{pluto}
- u'' -  \vec h_i\cdot{D^2\vec m}(\alpha+\vth_i)u =0 \mbox{\ in $(0,1),$} \quad u(0)=0, \ u'(1)=0,
\end{equation}
has only the null solution. This fact, however, does {\em not} directly imply surjectivity of \eqref{pluto}, for which we need to invoke the theory of Sturm-Liouville operators. Such theory also guarantees that the eigenvalues of \eqref{pippo} are discrete (cf. Theorem \ref{baseSL}). Therefore the current formulation of $(i)$ highlights the fact that, besides non-generic ``resonant'' cases, any minimizer $(\vh,\alpha,\vecth)$ is a solution to the Lagrangian system \eqref{EQ}. Note also that in such non-generic cases a function $\alpha$ which does not satisfy $\alpha'(0)=0$ is precluded from being the design of a minimizer if $\eps>0$. In addition, any minimizer is a solution to \eqref{EQ} if \eqref{bound-on-H} holds.}
\end{remark}
}

\begin{remark}{\rm{
As an immediate consequence of \ch{Theorem} \ref{min}, the bound \eqref{eq:13}, and \ch{Theorem} \ref{lagmul2}, we obtain that if \ch{$\eps>0$} and $\gamma>\overline\Theta^2 c_p^4$ then there exists a solution $(\vh,\alpha,\vecth,\bm{\lambda})\in\mathcal{H}\times\hl^n$ to system \eqref{EQ} such that $(\vh,\alpha,\vecth)$ is a minimizer of $\mathcal{C}$ in $\mathcal{A}$.
}}\end{remark}

The existence of a Lagrange multiplier justifies the approach proposed in \cite{ciambella2019form} to numerically approximate the minimizer of $\mathcal C$, which is based on \eqref{EQ}. In Section \ref{sec:uniqueness} we prove by a contraction argument that, at least for $\gamma$  sufficiently large, System \eqref{EQ} has a unique solution (see Proposition \ref{prop-step1+2}). As a by-product, we have:
\begin{theorem}
\label{exun}
Let $\tbar\in C([0,1])^n$, $\varepsilon>0$, and let $K>0$ such that
\begin{equation}\label{eq:15}
K<c_p^{-2}.
\end{equation}
Then exists $\gamma_*=\gamma_*(\tbar,\varepsilon,K)$ such that for every $\gamma> \gamma_*$ there exists a unique solution of system \eqref{EQ} within the following set:
\begin{equation}
\label{cond}
(\vh,\alpha,\vecth,\bm{\lambda})\in\mathcal{H}\times\hl^n \quad\mbox{such that}\quad \HH\leq K<c_p^{-2}.
\end{equation}
\ch{
Moreover,
\begin{equation}\label{new1}
\|\vartheta_i\|_\infty\le |\vec h_i|\quad\mbox{for all $i=1,\dots,n$.}
\end{equation}
}
\end{theorem}

\ch{Theorem} \ref{min}, \ch{Theorem} \ref{lagmul2} and Theorem \ref{exun} combine \ch{into:}
\begin{theorem}\label{coro-un}
Let $\tbar\in C([0,1])^n$, $\varepsilon>0$, and let $K>0$ such that \eqref{eq:15} holds. Then there exists $\gamma_{**}=\gamma_{**}(\tbar,\varepsilon,K)$ such that for any $\gamma>\gamma_{**}$ the minimizer in \ch{Theorem} \ref{min} is unique. Furthermore, it coincides with the unique solution to \eqref{EQ}, whence it is smooth and such that $\alpha'(0)=0$, \ch{and \eqref{new1} holds.}
\end{theorem}

\begin{proof}
Let $(\vh^{(j)},\alpha^{(j)},\vecth^{(j)})$, $j=1,2$ be two minimizers. Let $\gamma> \overline\Theta^2/K^2$. By \eqref{eq:13} in \ch{Theorem} \ref{min}, both minimizers satisfy
\begin{equation}\label{JJJ}
\left(\max_{i=1,\dots,n}|\vec h_i^{(j)}|\right)^2\leq \frac{\overline{\Theta}^2}{\gamma} < K^2<c_p^{-4},\qquad j=1,2.
\end{equation}
In particular, \eqref{bound-on-H} holds for both. Hence, by \ch{Theorem} \ref{lagmul2}, there exist $\bm{\lambda}^{(j)}\in \hl^n$ such that  $(\vh^{(j)},\alpha^{(j)},\vecth^{(j)},\bm{\lambda}^{(j)})\in\mathcal{H}\times\hl^n$ are solutions to system \eqref{EQ}. Assume in addition that $\gamma> \gamma_*(\overline{\bm\vartheta},\eps, K)$. Then, by \eqref{JJJ} and Theorem \ref{exun}, the two quadruplets, whence the two minimizers, coincide: therefore the proof is complete by choosing $\gamma_{**}=\max\{\overline{\Theta}^2/K^2, \gamma_*(\overline{\bm\vartheta},\eps,K)\}$.
\end{proof}

\ch{Theorem} \ref{coro-un} states that for $\gamma>\gamma_{**}$ the minimum is unique and may be numerically approximated by solving the Euler-Lagrange system \eqref{EQ} (hence, not necessarily by a direct approach, although the latter is used to prove the existence of the minimum). In fact, it is through the uniqueness of the solution of the Euler-Lagrange system that we are able to assert the uniqueness of the minimum.

\ch{
\begin{remark}{\rm
While Theorem \ref{coro-un} holds for any target $\overline\vecth$ (even very large ones), the minimizing state $\bm\vartheta$ will anyway be such that $\|\vartheta_i\|_\infty \le |\vec h_i| <c_p^{-2}$ for all $i=1,\dots,n$ (see \eqref{new1}). Now, $c_p^{-2}>3\pi/4$ is still a relevant value of the maximal rotation. However, one should bear in mind that large values of $\gamma$ may turn into minimizers with $|\vec h_i|$ much smaller than $c_p^{-2}$, and thus insufficient to drive the attained shapes close to the targets. This means that, {\em under the uniqueness conditions of Theorem \ref{coro-un}}, minimizing states may turn out to be far away from the targets when the latter ones are ``large'', a disappointing result from the point of view of engineering applications. In this respect, see also the comments to (a) and (b) below.
}\end{remark}
}

\ch{
\subsection{\ch{Open problems}} \label{sec:open}
We view the results in Section 1.5 as first steps in the mathematical analysis of the mechanical system under consideration. Indeed, quite a few relevant and interesting challenges are left open.

\smallskip

The first one is the existence of minimizers in the absence of penalization terms, i.e., with $\eps=\gamma=0$. \ch{While this is obvious in the non-generic case of attainable targets (see Remark \ref{epsgam}), it otherwise appears to be a nontrivial problem.  In fact, for generic targets,} it might even be that the control-design minimization problem \eqref{main-problem} is not well posed if $\eps=\gamma=0$ and $n>1$ (see Remark \ref{epsgam}).

\smallskip

The second one concerns {\em uniqueness} of minimizers, on which our results are admittedly limited by two conditions:
\begin{itemize}
\item[(a)] a moderate maximal intensity of the applied field ($|\vh|<c_p^{-2}=\pi^2/4$, see \eqref{eq:15}-\eqref{cond});
\item[(b)] a possibly large penalization constant ($\gamma>\gamma_{**}$).
\end{itemize}
We do not know whether (a) is optimal or not for the full design-control problem \eqref{main-problem}. However, as detailed in Remark \ref{lyasch}, we know that $c_p^{-2}$ is optimal for the state equation $(P_\vartheta)$. On the other hand, we believe that (b) is mainly technical (see Remark \ref{rem:technical}), and that uniqueness may hold even for values of $\gamma$ which are much smaller than $\gamma_{**}$. Improving the current bound would require a refinement of the estimates of $\vecth$ in terms of $\overline\vecth$ in the Lagrangian formulation, a challenging but important goal for further developments (see also Remark \ref{rem-residuals}).

\smallskip

Still related to uniqueness, Theorem \ref{lagmul2} shows that, besides non-generic cases, minimizers are critical points of $\mathcal L$. We expect that, above a certain threshold (being it $c_p^{-2}$ or larger), multiple critical points of $\mathcal L$ will emerge: it would be very interesting to develop selection criteria for identifying absolute minimizer(s) among multiple critical points.

\smallskip

The last major open question is of a different nature, and concerns the possibility of passing from a ``static'' to a dynamic framework, in which the rod moves in time following a prescribed path. In this framework, the $n$ targets would represent discrete snapshots of such continuous movement.
}

Further \ch{remarks} are presented in the concluding Section \ref{sec:concl}.

\section{Notation and preliminaries}\label{sec:preli}

In this section we introduce some notation as a complement to that already defined in the Introduction, and we collect preliminary results that will be needed in our subsequent developments. Other standard results are contained in the Appendix.

\smallskip

Given a vector $\vec{v}=(v_x,v_y)\in\mathbb R^2$, we let $|\vec{v}|=\sqrt{v_x^2+v_y^2}$ be its Euclidean norm and, for $\vec w$ another vector, we let $\vec{v}\cdot\vec{w}=v_xw_x+v_yw_y$ be the scalar product between $\vec v$ and $\vec w$. Given a list of vectors $\bm{\vec v}=(\vec v_1,\ldots, \vec v_n)$, with $\vec{v}_i\in\R^2$ for $i=1,\dots,n$, we let $|\bm{\vec v}|:=|\vec{v}_1|+\dots+|\vec{v}_n|$. For $f:I\to \R$ a measurable function, we use the abbreviation $\|f\|_{p}\equiv\|f\|_{L^p(I)}$ for all exponents $p\ge 1$. We use similar abbreviations for measurable vector-valued functions.
We recall that, by the Sobolev embedding theorem, $\hl\subset C([0,1])$, where $C([0,1])$ is the space of the continuous functions on $[0,1]$. We record for later use the inequality
\begin{equation}
\label{def-S}
\|v\|^2_{\infty}\leq \io (v')^2 \quad\mbox{for all $v\in\hl$}
\end{equation}
which is sharp, as can be seen by taking $v(x)=x$.

\smallskip

We denote by $c_p=2/\pi$ the \emph{best constant} in the \emph{Poincar\'e-type inequality}:
\begin{equation}\label{def-c}
\io v^2 \le c_p^2\io (v')^2 \quad \mbox{for all $v\in \hl$}.
\end{equation}
It follows from \eqref{def-c} and from the definition \eqref{defhl} that
\begin{equation}\label{eq:10}
  \|v\|^2:=\io (v')^2
\end{equation}
is equivalent to the  Sobolev norm on $\hl$. Accordingly, we henceforth shall use the norm \eqref{eq:10} to endow $\hl$ with a Hilbert-space structure. For $(\vh,\alpha,\vecth)\in\mathcal{H}$ we define $\displaystyle \|(\vh,\alpha,\vecth)\|_{\mathcal{H}}=|\vh|+\|\alpha\|+\|\vecth\|$. It follows that $\mathcal{H}$ is an Hilbert space.

\smallskip

For $(\mathcal X,\|\cdot\|_{\mathcal X})$, $(\mathcal Y,\|\cdot\|_{\mathcal Y})$ Banach spaces we denote by $\mathcal L(\mathcal X,\mathcal Y)$ the space of bounded linear operators from $\mathcal X$ to $\mathcal Y$, and we let $\|\cdot\|_{\mathcal L(\mathcal X,\mathcal Y)}$ be the operator norm. Moreover, we write $\langle \cdot,\cdot \rangle_{{\mathcal X',\mathcal X}}$ to denote the pairing between a Banach space $\mathcal X$ and its dual: in fact, we will omit the indexing whenever the space $\mathcal X$ is clear from the context.

\smallskip

If not otherwise specified, we will denote by $C$ a generic constant whose value may possibly change within the same chain of inequalities, and by $C(\cdot)$ constants whose value only depend on the parameters and variables listed within parentheses.
\smallskip

Finally, we observe that the function $\vec m:\mathbb R\to\mathbb R^2$ defined in \eqref{eq:17} is bounded, infinitely differentiable and its $N$-th derivative, defined consistently with \eqref{consist}, is
\begin{equation}\label{eq:18}
  D^N\vec m(v)=\left({\begin{smallmatrix}
    0 &-1\\[0.5em]
    +1 & 0
  \end{smallmatrix}}\right)^N \vec m(v),\qquad\text{for all }N\in\mathbb N,
\end{equation}
namely, $D^N\vec m(v)$ is the vector obtained by rotating $\vec m(v)$ in the counter-clockwise direction by the amount $N\pi/2$. Thus,
\begin{equation}\label{eq:16}
  \bigl|D^N\vec m(v)\bigr|=1,\quad\text{for all }v\in\R \text{ and }N\in\mathbb N.
\end{equation}
Hence,
\begin{equation*}
  |D^N\vec m(v_1)-D^N\vec m(v_2)|\le \int_{v_1}^{v_2} \lvert D^{N+1}\vec m(v)\rvert dv=|v_1-v_2|\quad\text{for all }v_1,v_2\in\mathbb R.
\end{equation*}
As a consequence of this observation, we record three bounds which will be used several times.
\begin{lemma}
\label{trig1}
Let $(\vec{h},\alpha,\vartheta,\lambda),(\vec{\tilde{h}},\tilde{\alpha},\tilde{\vartheta},\tilde{\lambda})\in\R^2\times \R\times \R\times \R$ and $N\in\mathbb N$. Then
\begin{subequations}
\begin{eqnarray}
  \label{trigR4}
\left|\lambda D^N\vec m(\alpha+\vth)- \tilde\lambda D^N\vec m(\tilde\alpha+\tilde\vartheta) \right| & \le &  |\lambda-\tilde\lambda| + |\tilde{\lambda}| \left(|\alpha-\tilde{\alpha}|+|\vartheta-\tilde{\vartheta}|\right),
\\
  \label{trigR2}
\left|\vec h\cdot D^N\vec m(\alpha+\vth)- \tilde{\vec h}\cdot D^N\vec m(\tilde\alpha+\tilde\vartheta) \right| & \le &  |\vec h-\vec{\tilde h}| + |\vec{\tilde h}| \left(|\alpha-\tilde{\alpha}|+|\vartheta-\tilde{\vartheta}|\right),
\\
\nonumber
\left|\lambda \vec h\cdot D^N\vec m(\alpha+\vth)- \tilde\lambda \vec{\tilde h}\cdot D^N\vec m(\tilde\alpha+\tilde\vartheta) \right| & \le & |\vec{h}| |\lambda-\tilde\lambda| + |\tilde \lambda| |\vec{h}-\vec{\tilde h}|
\\
&& + |\vec{\tilde{h}}||\tilde{\lambda}| \left(|\alpha-\tilde{\alpha}|+|\vartheta-\tilde{\vartheta}|\right).\label{trigR}
\end{eqnarray}
\end{subequations}
\end{lemma}

\section{Existence of a minimizer}\label{sec:minimizer}

In this section we address the existence of a minimizer to the optimal control-design problem \eqref{main-problem}. 

\begin{proof}[Proof of \ch{Theorem} \ref{min}]
We recall that the admissible set is defined in \eqref{eq:11}. We begin by noting that $\ch{(\bm{\vec{0}},0,\bm{0})}\in\mathcal{A}$, hence $\mathcal{A}$ is not empty.
Next, we let
$$
m=\inf_\mathcal{A} C(\vh,\alpha,\vecth),
$$
and we consider a minimizing sequence, i.e. a sequence $\{(\vh_k,\alpha_k,\vecth_k)\}\subset\mathcal{A}$ with $\vh_k=(\vec{h}_{k1},\dots,\vec{h}_{kn})$, $\vec{h}_{ki}=(h_{kix},h_{kiy})$, and $\vecth_k=(\vth_{k1},\dots,\vth_{kn})$, such that $\mathcal{C}(\vh_k,\alpha_k,\vecth_k)\to m$ as $k\to +\infty$.  In particular, by the definition of $\mathcal{C}$, a constant $C$ exists such that
\begin{equation}\label{app-1}
 \sum_{i=1}^n \io |\vth_{ki}|^2 + \io |\alpha_k'|^2 + \sum_{i=1}^n|\vec{h}_{ki}|^2 \leq C
\end{equation}
for all $k\in \N$. Moreover $\vth_{ki}$ satisfies
\begin{equation}
\label{app2}
\io \left(\vth_{ki}' v'- \vec h_{ki}\cdot{D\vec m}(\alpha_k+\vth_{ki})v\right) =0, \qquad \forall v\in\hl
\end{equation}
for all $k\in \N$ and $i=1,\dots,n$. Choosing $\vth_{ki}$ as test function in \eqref{app2} and recalling \eqref{app-1}, we obtain
\begin{equation*}
\|\vth_{ki}\|\leq 2C \quad\mbox{for all $k\in \N$ and every $i=1,\dots,n$.}
\end{equation*}
Hence, by a standard  compactness argument, $(\vh,\alpha,\vecth)\in \mathcal{H}$ exists such that, by passing to a subsequence (not relabeled),
\begin{equation}
\label{conv}
\begin{array}{l}
\vec{h}_{ki} \rightarrow \vec{h}_i \text{ in } \R^2, \text{ for every } i=1,\dots,n , \\
\alpha_k \rightarrow \alpha \text{ weakly in } \hl \text{ and uniformly in } C([0,1]), \\
\vth_{ki} \rightarrow \vth_i \text{ weakly in } \hl \text{ and uniformly in } C([0,1]), \text{ for every } i=1,\dots,n.
\end{array}
\end{equation}
Letting $k$ tend to infinity in \eqref{app2} and using the convergence statement \eqref{conv}, we conclude that $\vth_i$ is a weak solution of \eqref{pthi} for every $i=1,\dots,n$, so $(\vh,\alpha,\vecth)\in\mathcal{A}$. Moreover, by lower semi-continuity,
$$
\mathcal{C}(\vh,\alpha,\vecth)\leq \liminf_{k\to\infty}\mathcal{C}(\vh_k,\alpha_k,\vecth_k)=m.
$$
This implies that $(\vh,\alpha,\vecth)$ is a minimizer of $\mathcal{C}$. In addition, since $(\bm{0},0,\bm{\vec{0}})\in\mathcal{A}$, for any minimizer we have
\begin{equation*}
\frac{\gamma}{2}\sum_{i=1}^n |\vec{h}_i|^2 \le \mathcal{C}(\vh,\alpha,\vecth)\leq \mathcal{C}\ch{(\bm{\vec{0}},0,\bm{0})} = \frac12\sum_{i=1}^n\ch{\int_0^1}|\overline{\vartheta}_i|^2,
\end{equation*}
which implies \eqref{eq:13}. \ch{This proves part $(i)$ of the theorem.

\medskip

We now restrict attention to attainable targets: that is, we assume that $\tbar$ is such that $(\vhb,\bar\alpha,\tbar)\in \mathcal A$ for some  $\vhb\in\R^{2n}$  and some $\bar \alpha\in H^1_{0L}(I)$. Note that in this case $\mathcal C_{0,0}(\vhb,\bar\alpha,\tbar)=0$, so $(\vhb,\bar\alpha,\tbar)$ is a minimizer of $\mathcal{C}_{0,0}$. Let $(\vh_\eps,\alpha_\eps,\vecth_\eps)$ be a minimizer of $\mathcal C_{\eps,\eps}$ in $\mathcal A$. We have $\mathcal C_{\eps,\eps} (\vh_\eps,\alpha_\eps,\vecth_\eps)\le \mathcal C_{\eps,\eps}(\vhb,\bar\alpha,\tbar)$, that is,
$$
\sum_{i=1}^n \io |\vth_{\eps i}-\vthb_i|^2 + \varepsilon\io |\alpha_\eps'|^2 + \eps \sum_{i=1}^n |\vec{h}_{\eps i}|^2 \le \varepsilon \io |\bar\alpha'|^2 +\eps \sum_{i=1}^n |\vec{\bar h}_i|^2.
$$
This means that
$$
\sum_{i=1}^n \io |\vth_{\eps i}-\vthb_i|^2 \stackrel{\eps\to 0}\to 0 \quad\mbox{and}\quad \io |\alpha_\eps'|^2 + \sum_{i=1}^n |\vec{h}_{\eps i}|^2 \le
\io |\bar\alpha'|^2 + \sum_{i=1}^n |\vec{\bar h}_i|^2.
$$
Therefore, arguing as above, we see that for a subsequence $\vh_\eps\to \vh_0$ in $\R^{2n}$, $\alpha_\eps\rightharpoonup \alpha_0$ in $H^1_{0L}(I)$, $\vecth_\eps\rightharpoonup  \vecth_0=\tbar$ in $H^1_{0L}(I)^n$, and $(\vh_0,\alpha_0,\tbar)\in \mathcal A$. In addition, it is obvious that $\mathcal C_{0,0}(\vh_0,\alpha_0,\tbar)=0$. Therefore $(\vh_0,\alpha_0,\tbar)$ is a minimizer of $\mathcal C_{0,0}$ in $\mathcal A$.

\medskip

In order to prove $(iii)$, fix $\bar\vartheta \in H^3(I)\subset C^2([0,1])$ with $\bar\vartheta(0)=0$ and $\bar\vartheta'(1)=0$. Let $H > \max_{s\in[0,1]}|\overline\vartheta''(s)|$ and set
\begin{equation*}
\vec{\bar h}=(H\cos\psi,H\sin\psi),\quad \bar\alpha(s)=\arcsin\Big(\frac{\overline\vartheta''(s)}H\Big)-\overline\vartheta(s)+\psi_0,
\end{equation*}
Hence, choosing $\psi=-\arcsin\Big(\frac{\overline\vartheta''(0)}H\Big)$, we have $\alpha(0)=0$. Moreover, since $\bar\vartheta'' \in H^1(I)$ and $\arcsin$ is a Lipschitz continuous function in any compact subset of $(-1,1)$, we deduce that $\bar\alpha\in H^1_{0L}(I)$. This implies that $(\vec{\bar h}, \bar\alpha,\bar\vartheta)\in \mathcal H$. \\
Finally, after straightforward computations using angle sum identities, one sees that
\begin{equation*}
\begin{cases}
-\bar \vth''-\vec{\bar h}\cdot {D\vec m}(\bar\alpha+\bar\vartheta)=0   & \mbox{in $I$}, \\
\bar\vth(0)=0, \\
\bar\vth'(1)=0.
\end{cases}
\end{equation*}
This shows that $\bar\vartheta$ is always attainable.
}
\end{proof}

\section{The basic equation}\label{sec:basic}

The next sections will be devoted to the analysis of the Lagrange-multiplier system \eqref{EQ}. Its equations, \1eq, \2eq, and $({P}_{\alpha})$, share the following structure:

\begin{equation}
\label{general}
\begin{cases}
-v''+f(s,v)=0 & \mbox{in $(0,1),$} \\
v(0)=0, \\
v'(1)=0.
\end{cases}
\end{equation}
In particular, with reference to \eqref{EQ}, we have that
\begin{equation}\label{eq:12}
 \text{\eqref{general} is equivalent to: }\left\{ \begin{aligned}
    &({P}_{\vth_i})\text{ for }f(s,v)=-\vec h_i\cdot D\vec m(\alpha(s)+v);\\
    &({P}_{\lambda_i})\text{ for }f(s,v)=-v\vec h_i\cdot D^2\vec m(\alpha(s)+\vth_i(s))+\overline\vth_i(s)-\vth_i(s);\\
    &({P}_{\alpha})\text{ for }f(s,v)=\frac 1 \varepsilon \sum_{i=1}^n\lambda_i(s)\vec h_i\cdot D^2\vec m(v+\vth_i(s)).
  \end{aligned}\right.
\end{equation}
\begin{defin}
\label{weaksol}
Let $f\in L^1(I\times \R)$. A function $v$ belonging to $\hl$ is a (weak) solution to problem \eqref{general} if
\begin{equation}\label{weak-form}
\io v'w'+\io f(s,v)w=0, \qquad \text{for all } w\in\hl.
\end{equation}
\end{defin}

In the following Lemma we provide (to the extent we need) uniqueness, existence, and boundedness results for solutions of \eqref{general}.

\begin{lemma}
\label{qualit} Let $f\in L^\infty(I\times \R)$, let $L\in (0,c_p^{-2})$ be such that
\begin{equation}\label{Lip-f}
|f(s,v_1)-f(s,v_2)| \le L |v_1-v_2| \quad\mbox{for a.e. $s\in I$ and for all $v_1,v_2\in \R$,}
\end{equation}
and let $c_p$ be defined by \eqref{def-c}. Then there exists a unique solution $v$ to problem \eqref{general} in the sense of Definition \ref{weaksol}, and this solution satisfies the bounds
\begin{equation}
\label{bound2}
\displaystyle \|v\|\le \frac{c_p}{1-Lc_p^2}\|f(s,0)\|_\infty,\quad \|v\|_{\infty}\le \|f\|_{\infty}.
\end{equation}
\end{lemma}
\begin{proof}
Let $f_0(s)=f(s,0)$. For $ g(s,v)=\int_0^v f(s,t) dt$, we let:
$$
\mathscr F(w)=\frac12 \io w'^2+\io g(s,w), \qquad w\in\hl.
$$
This position defines a G\^ateaux-differentiable, weakly-lower semicontinuous functional $\mathscr F:\hl\to\mathbb R$. Since
$$
|g(s,v)|\le \int_0^v |f(s,t)|dt\stackrel{\eqref{Lip-f}}\le |f_0(s)||v|+\frac12 L|v|^2,
$$
we have
\[
  \mathscr F(w)\ge \frac12 \io {w'}^2-\io |f_0 w|-\frac12 L\io w^2\ge \frac12(1-Lc_p^2)\|w\|^2-\|f_0\|_2\|w\|_2,
\]
whence, by \eqref{def-c},
\[
  \mathscr F(w)\ge \frac12(1-Lc_p^2)\|w\|^2-c_p\|f_0\|_\infty\|w\|.
 \]
 This inequality implies that $\mathscr F$ is \emph{coercive}, thanks to the hypothesis $L<c_p^{-2}$. The coercivity and the lower semicontinuity of $\mathscr F$ imply, by a standard argument, that $\mathscr F$ has a minimizer $v$ in $\hl$ (see \cite{Giusti2003}). Since $\mathscr F$ is G\^ateaux differentiable, $v$ is also a weak solution of Problem \eqref{general}.

\smallskip

In order to prove uniqueness, let $v_1$ and $v_2$ be two weak solutions of \eqref{general}. According to Definition \ref{weaksol}, $v_1-v_2$ is a legal test function for the weak formulation of \eqref{general}. We use this test in \eqref{weak-form}. On taking the difference between the resulting equations we obtain:
$$
\io [(v_1-v_2)']^2 + \io (f(s,v_1)-f(s,v_2))(v_1-v_2)=0.
$$
It follows from the assumption on $f$ and from the Poincar\'e inequality \eqref{def-c} that
\[
  \io (f(s,v_1)-f(s,v_2))(v_1-v_2) \le L\io |v_1-v_2|^2\le L c_p^2\|v_1-v_2\|^2,
\]
whence
$$
(1-Lc_p^2)\|v_1-v_2\|^2\leq 0,
$$
and thence $v_1=v_2$, given that $Lc_p^2<1$.

\smallskip

Taking $v$ as test function in \eqref{weak-form} we obtain that
\begin{eqnarray*}
\|v\|^2=\io v'^2 &=& -\io f(s,v)v \le \io |f_0(s)||v|+ L\io |v|^2
\\ &\le & \|f_0\|_\infty\left(\io v^2\right)^{1/2} + L\io |v|^2
\stackrel{\eqref{def-c}}\le
 c_p\|f_0\|_\infty\|v\| + Lc_p^2 \|v\|^2,
\end{eqnarray*}
whence the first bound in \eqref{bound2}. Finally, the second bound in \eqref{bound2} is immediate from the representation formula
$$
v(s)= \int_0^s \int_{s'}^1 f(s'',v(s''))d s''\ ds'.
$$
\end{proof}

\begin{remark}[Regularity and boundary values of the solution to Problem \eqref{general}]{\rm
\label{regularity}
\ch{Under the assumption $f\in L^\infty(I\times \R)$ of Lemma \ref{qualit}}, we note that if $v$ is a weak solution to problem \eqref{general}, then
$$
v\in H^2(I)=\left\{w\in L^2(I) : w',w''\in L^2(I)\right\}.
$$
The Sobolev embedding theorem
(see for instance Sec. 2.1 of \cite{buttazzo}) implies that $v\in C^1([0,1])$, and that the boundary conditions are satisfied pointwise. Indeed, since $v\in\hl$, we have that $v(0)=0$. Moreover, multiplying by an arbitrary function $w\in\hl$ and integrating in $I$ equation \eqref{general} we obtain
\begin{equation}\label{regu}
\io -v''w+\io f(s,v)w=0 \qquad \text{for all } w\in\hl.
\end{equation}
Integrating by parts the first term of the l.h.s. of \eqref{regu} we have
$$
v'(1)w(1)=\io (v'w)'=\io v'w'+\io f(s,v)w\stackrel{\eqref{weak-form}}= 0 \qquad \text{for all } w\in\hl.
$$
This implies that $v'(1)=0$.
}\end{remark}

As a by-product of the previous discussion, we obtain:
\begin{coro}\label{cor:theta}
Let $\vec h\in \R^2$ and \ch{$\alpha: I\to \R$ measurable}. Then the energy functional $\mathcal E$ defined by \eqref{eq:6} has a minimizer $\vartheta\in H^1_{0L}$. Furthermore, $\vartheta$ solves \eqref{pthi1} and \ch{it} is unique if \eqref{eq:8} holds. \ch{Finally, if $\alpha\in H^1(I)$, then $\vartheta\in H^3(I)$.}
\end{coro}

\begin{proof}
Existence of a minimizer of $\mathcal E$ in $\hl$ can be proved with the same arguments used in Section \ref{sec:minimizer}. By standard variational considerations, $\vartheta$ is a solution to \eqref{weak-form} with $f(s,\vartheta)= -\vec h\cdot D\vec m(\alpha(s)+\vartheta)$: hence, by Lemma \ref{qualit}, it is unique if $|\vec h|<c_p^{-2}$. \ch{Moreover, if $\alpha\in H^1(I)$, we deduce that $\vec h\cdot D\vec m(\alpha(s)+\vartheta)\in H^1(I)$. Hence, by \eqref{pthi1}, $\vartheta''\in H^1(I)$, i.e. $\vartheta\in H^3(I)$.
}
\end{proof}

\ch{
  \begin{remark} \label{lyasch}
    {\rm
      The condition $|\vec h|<c_p^{-2}$ \ch{in Corollary \ref{cor:theta}} is optimal in general, as the following counterexample shows. Taking $\alpha=0$ and $\vec h=(-H,0)$ in \eqref{pthi1} yields
\begin{equation}
\begin{cases}
\label{pthi12}
-\vth''-H\sin\ch{\vth}=0   & \mbox{in $(0,1),$} \\
\vth(0)=0, \ \vth'(1)=0.
\end{cases}
\end{equation}
Problem \eqref{pthi12} is the same one which governs a cantilever under a compressive thrust applied at its free end. It admits the trivial solution $\vth=0$ for every $H\in\mathbb R$. This solution is unique for $H<c_p^{-2}=\pi^2/4$. However, a non-trivial branch emanates from the singular point $(\vartheta,H)=(0,c_p^{-2})$, and uniqueness is lost for $H>c_p^{-2}$. Indeed, after integration, a strictly increasing solution $\vth(s)$ with $\vth(1)=\vth_1>0$ is implicitly given by
\begin{equation}\label{kk}
\int_0^{\vth(s)} \frac{d t}{\sqrt{\cos t -\cos \vth_1}} =\sqrt{2H} s, \quad \sqrt{2H}=\int_0^{\vth_1} \frac{d t}{\sqrt{\cos t -\cos \vth_1}} =\sqrt{2} K(\sin\tfrac{\vth_1}{2}),
\end{equation}
where $K(k):=\int_0^{\pi/2}\frac{d \phi}{\sqrt{1-k^2\sin^2\phi}}$ is the complete elliptic integral of the first kind (for the last equality in \eqref{kk}, one uses the change of variables $\sin\phi=\sin\tfrac t2/\sin\tfrac{\vth_1}{2}$). It is easily checked that $K(k)$ is increasing, $K(k)\to \frac{\pi}{2}$ as $k\to 0^+$, and $K(k)\to +\infty$ as $k\to 1^-$: hence for any $H>\pi^2/4=c_p^{-2}$ the second equation in \eqref{kk} has a solution $\theta_1(H)$, and inverting the first one we obtain $\vth(s)$. We aside mention that, when $\vec h=(0,H)$ is taken instead of $\vec h=(-H,0)$, numerical evidence in \cite{Ciambella2018} suggests uniqueness of solutions to \eqref{pthi1} for values of $H$ substantially larger than $c_p^{-2}$.
}
\end{remark}
}

\ch{
\begin{remark}
    \rm One may wonder whether, under the condition $|\vec h|<c_p^{-2}$, the discrepancy between the full nonlinear theory and a simpler, linearized theory for \eqref{pthi1} would be negligible, so as to motivate an employment of the latter to simplify computations. In this respect, we remark that a simple numerical computation of the equilibrium shape for $\vec h=(0,H)$ and $H$ close to $c_p^{-2}$ yields a discrepancy between nonlinear and linear theory of nearly $40\%$.
\end{remark}
}

\section{The Lagrange multiplier formulation}\label{sec:lagr2}

We recall the definition \eqref{defg-intro} of the constraint mapping:
\begin{equation}
\label{defg}
\langle G(\vh,\alpha,\vecth),\bm{u}\rangle =\sum_{i=1}^n\left\{\io \vth_i'u_i'-\io \vec h_i\cdot{D\vec m}(\alpha+\vth_i) u_i\right\}.
\end{equation}
Since $\displaystyle |\langle G(\vh,\alpha,\vecth),\bm{u}\rangle|\leq C(\vh,\vecth)\|\bm{u}\|$ for every $\bm{u}\in\hl^n$, $G(\vh,\alpha,\vecth)$
is a linear bounded functional. Thus \eqref{defg} defines a map $G:\mathcal{H}\to(\hl^n)'$. Thanks to the equivalence
\begin{equation*}
(\vh,\alpha,\vecth)\in\mathcal{A} \quad \Leftrightarrow \quad G(\vh,\alpha,\vecth)= 0,
\end{equation*}\color{black}
we can write
\begin{equation}
\label{defA}
\mathcal{A}\stackrel{\eqref{eq:11}}=\left\{(\vh,\alpha,\vecth)\in\mathcal{H} : G(\vh,\alpha,\vecth)= 0\right\}.
\end{equation}
Proposition \ref{lag} in the Appendix of this paper provides sufficient conditions for the existence of the Lagrange multiplier $\bm\lambda$. We are going to use this proposition as a tool to characterize the minimizers of $\mathcal{C}$ in $\mathcal{A}$. To this aim, we need to assess the regularity of the functional $\mathcal C$ and of the operator $G$; the next statement concerns their Fr\'echet differentiability, which we shall obtain as a consequence of Proposition \ref{equiv} and the following lemma.
\begin{lemma}
\label{lemc1}
\ch{Let $\eps,\gamma\ge 0$.} The operators $\mathcal{C}:\mathcal{H}\to\R$ and $G:\mathcal{H}\to(\hl^n)'$ \ch{are $C^1$,} with $D\mathcal{C}:\mathcal{H}\to\mathcal{H}'$  and $DG:\mathcal{H}\to\mathcal{L}(\mathcal{H},(\hl^n)')$ being represented by
\begin{equation}
\label{DC}
\displaystyle D\mathcal{C}(\vh,\alpha,\vecth)(\bm{\vec{k}},\beta,\bm{\iota})=\gamma\sum_{i=1}^n \vec{h}_i\cdot\vec{k}_i+\varepsilon\io \alpha'\beta'+\sum_{i=1}^n\io (\vth_i-\vthb_i)\iota_i,
\end{equation}
respectively
\begin{eqnarray}
\label{DG}
\nonumber
\lefteqn{\langle DG(\vh,\alpha,\vecth)(\bm{\vec{k}},\beta,\bm{\iota}),\bm{u}\rangle \ = \ \sum_{i=1}^n\io -\vec{k}_{i}\cdot D\vec{m}(\alpha+\vth_i)u_i}
\\ &&+\!\sum_{i=1}^n\!\io\!\left\{ \iota_i'u_i'-\vec h_i\cdot\left({D^2\vec m}(\alpha+\vth_i)\iota_i+D^2\vec m(\alpha+\vth_i)\beta\right)\!u_i\!\right\}\!\!,
\end{eqnarray}
for every $(\vh,\alpha,\vecth),(\bm{\vec{k}},\beta,\bm{\iota})\in\mathcal{H}$ and $\bm{u}\in\hl^n$.
\end{lemma}

\begin{proof}
Fix $\bm{\varphi}:=(\vh,\alpha,\vecth)\in \mathcal{H}$. We consider a sequence $\{\bm{\varphi}_k\}:=\{(\vh_k,\alpha_k,\vecth_k)\}$, with $\vh_k=(\vec{h}_{k1},\dots,\vec{h}_{kn})$, $\vec{h}_{ki}=(h_{kix},h_{kiy})$, and $\vecth_k=(\vth_{k1},\dots,\vth_{kn})$ such that $\|\bm{\varphi}_k- \bm{\varphi}\|_{\mathcal H}\rightarrow 0$ as $k\rightarrow +\infty$. In particular, $C$ exists such that
\begin{equation}
\label{normeq2}
|\vh_k|\le C.
\end{equation}
First we focus on $\mathcal{C}$. We trivially have $\mathcal{C}(\vh_k,\alpha_k,\vecth_k)\to \mathcal{C}(\vh,\alpha,\vecth)$ in $\R$, hence $\mathcal{C}$ is continuous. The G\^ateaux derivative $\mathcal C'(\bm{\varphi})$ can be computed explicitly via Definition \ref{gat}, and it coincides with the right-hand side of \eqref{DC}. In order to show that $\mathcal C$ is $C^1$, we write (using \ch{the} Cauchy-Schwarz inequality)
\begin{align*}
&|\left(\mathcal{C}'(\bm{\varphi})-\mathcal{C}'(\bm{\varphi}_k)\right)(\bm{\vec{k}},\beta,\bm{\iota})|
\\ &
\stackrel{\eqref{DC}}\leq \left\{\gamma\sum_{i=1}^n|\vec{h}_i-\vec{h}_{ki}|+\varepsilon\|\alpha-\alpha_k\|+\sum_{i=1}^n \|\vth_i-\vth_{ki}\|_{2}\right\}\|(\bm{\vec{k}},\beta,\bm{\iota})\|_{\mathcal{H}},
\end{align*}
hence
$$
\displaystyle \|\mathcal{C}'(\bm{\varphi})-\mathcal{C}'(\bm{\varphi}_k)\|_{\mathcal H'}\leq C\|\bm{\varphi}-\bm{\varphi}_k\|_{\mathcal{H}}.
$$
Thus the G\^ateaux derivative $\mathcal C'$ of $\mathcal C$ is (Lipschitz) continuous with respect to the operator norm and, by applying Proposition \ref{equiv}, we conclude that $\mathcal C$ is Fr\'echet differentiable, that its differential is $D\mathcal C = \mathcal C'$ and that therefore $\mathcal{C}$ is $C^1$.

\medskip

Now we focus our attention on $G$, by first proving that $G$ is continuous. To this aim, we fix $\bm{u}=(u_1,\dots,u_n)\in\hl^n$ and we compute:
\begin{eqnarray}
\nonumber
  \langle G(\bm{\varphi})-G(\bm{\varphi}_k),\bm{u}\rangle&\stackrel{\eqref{defg}}=& \sum_{i=1}^n\io(\vth_i-\vth_{ki})'u_i'  -\sum_{i=1}^n\io \left(\vec h_i\cdot{D\vec m}(\alpha+\vth_i)- \vec{h}_{ki}\cdot D\vec m(\alpha_k+\vth_{ki})\right)u_i\nonumber
\\ \nonumber
&\stackrel{\eqref{trigR2}}\le& \|\vecth-\vecth_k\|\|\bm{u}\| +\sum_{i=1}^n \io\bigl(|\vec{h}_i-\vec{h}_{ki}|+|\vec{h}_{ki}|\left(|\alpha-\alpha_k|+|\vth_i-\vth_{ki}|\bigr)\right)|u_i|.
\end{eqnarray}
Then, by making use of H\"older and Poincar\'e inequalities, we deduce the inequality
$$
|\langle G(\bm{\varphi})-G(\bm{\varphi}_k),\bm{u}\rangle|\leq C \left((1+|\vh_k|)\|\vecth-\vecth_k\|+|\vh-\vh_k|+|\vh_k|\|\alpha-\alpha_k\|\right)\|\bm{u}\|,
$$
whence
\begin{eqnarray*}
\|G(\bm{\varphi})-G(\bm{\varphi}_k)\|_{(\hl^n)'} &\leq & C\left((1+|\vh_k|)\|\vecth-\vecth_k\|+|\vh-\vh_k|+|\vh_k|\|\alpha-\alpha_k\|\right)
\\ &\stackrel{\eqref{normeq2}}\leq  & C\|\bm{\varphi}-\bm{\varphi}_k\|_{\mathcal{H}}\to 0,
\end{eqnarray*}
hence $G$ is (Lipschitz) continuous. The G\^ateaux derivative $G'(\bm{\varphi}):\mathcal H\to (\hl^n)'$ can be computed explicitly via its definition, and it coincides with the right-hand side of \eqref{DG}.
Hence we deduce that
\begin{eqnarray*}
\lefteqn{\langle\left(G'(\bm{\varphi})-G'(\bm{\varphi}_k)\right)(\bm{\vec{k}},\beta,\bm{\iota}),\bm{u}\rangle} \\ &=&-\sum_{i=1}^n\io \vec{k}_i\cdot(D\vec m(\alpha+\vth_i)-D\vec m(\alpha_k+\vartheta_{ki}))u_i
\\
&&-\sum_{i=1}^n\io\left( \vec h_i\cdot
D^2\vec m(\alpha+\vartheta_i)-\vec h_{ki}\cdot D^2\vec m(\alpha_k+\vartheta_{ki})
\right)(\beta+\iota_i)u_i
\\
&\stackrel{\eqref{trigR2}}\leq &\sum_{i=1}^n \io |\vec{k}_i|(|\alpha-\alpha_k|+|\vth_i-\vth_{ki}|)|u_i| \\
&&+\sum_{i=1}^n \io \left(|\vec{h}_i-\vec{h}_{ki}|+|\vec{h}_{ki}|(|\alpha-\alpha_k|+|\vth_i-\vth_{ki}|)\right)|\beta+\iota_i||u_i| \\
&\stackrel{\eqref{normeq2}}\leq &C\left(\left(|\vh-\vh_k|+\|\alpha-\alpha_k\|+\|\vecth-\vecth_k\|\right)(|\bm{\vec{k}}|+\|\beta\|+ \|\bm{\iota}\|)\right)\|\bm{u}\|,
\end{eqnarray*}
where in the last inequality we have also used H\"older and Poincar\'e inequalities. It follows that
\begin{equation*}
\begin{aligned}
  \|(G'(\bm{\varphi})-G'(\bm{\varphi}_k))(\bm{\vec{k}},\beta,\bm\iota)\|_{(\hl^n)'} &=\sup_{\|\bm u\|_{\hl^n}=1}|\langle(G'(\bm{\varphi})-G'(\bm{\varphi}_k))(\bm{\vec{k}},\beta,\bm\iota),\bm u\rangle|\\
  &\leq C\|\bm{\varphi}-\bm{\varphi}_k\|_{\mathcal{H}}\|(\bm{\vec{k}},\beta,\bm\iota)\|_{\mathcal{H}},
  \end{aligned}
\end{equation*}
hence that
$$
\|G'(\bm{\varphi})-G'(\bm{\varphi}_k)\|_{\mathcal L(\mathcal H,(\hl^n)')}=\sup_{\|\bm{(\vec{k}},\beta,\bm\iota)\|_{\mathcal H}=1}\|(G'(\bm{\varphi})-G'(\bm{\varphi}_k))(\bm{\vec{k}},\beta,\bm\iota)\|_{(\hl^n)'} \leq C\|\bm{\varphi}-\bm{\varphi}_k\|_{\mathcal{H}}.
$$
This implies, applying Proposition \ref{equiv}, that $DG=G'$ and that $G$ is $C^1$.
\end{proof}

\ch{We may now prove Theorem \ref{lagmul2}.}

\begin{proof}[Proof of \ch{Theorem} \ref{lagmul2}]
\ch{Let $(\vh,\alpha,\vecth)$ be a minimizer of $\mathcal{C}$ in $\mathcal{A}$. We need} to prove that $DG(\vh,\alpha,\vecth)$ is surjective, that is, for every $T\in(\hl^n)'$ there exists $\bm{\psi}\in\mathcal{H}$ such that
\begin{equation}
\label{sur1}
\langle DG(\vh,\alpha,\vecth)(\bm{\psi}),\bm{u}\rangle =\langle T,\bm{u}\rangle \quad\mbox{for all $\bm{u}=(u_1,\dots,u_n)\in\hl^\ch{n}$.}
\end{equation}
\ch{It suffices to show that \eqref{sur1} has a solution $\bm{\psi}$ of the form ${\bm{\psi}}=(\bm{\vec{0}}, 0,\bm{v})$ for some} $\bm{v}\in\hl^n$. In this case, the l.h.s. of \eqref{sur1} defines a bilinear form, $a:\hl^n\times\hl^n\to\R$:
\begin{equation}
\label{sur2}
a(\bm{v},\bm{u}):= \langle DG(\vh,\alpha,\vecth)\ch{(\bm{\vec{0}}, 0,\bm{v})},\bm{u}\rangle \stackrel{\eqref{DG}}= \sum_{i=1}^n\io v_i'u_i'-\vec h_i\cdot D^2\vec m(\alpha+\vartheta_i)v_iu_i.
\end{equation}
\ch{
Applying Theorem \ref{teoSL} to each component of $DG$ (with $\delta=1$, $\tilde L=DG_i$, and $\tilde r=\vec h_i\cdot D^2\vec m(\alpha+\vartheta_i)$) we obtain $(i)$ and $(ii)$ of Theorem \ref{lagmul2}.
}

\smallskip

\ch{In order to prove $(iii)$, assume that $DG(\vh,\alpha,\vecth)$ is surjective.} In view of Lemma \ref{lemc1} and Proposition \ref{lag},
there exists a Lagrange multiplier $\bm{\lambda}\in \hl^n$ such that $(\vh,\alpha,\vecth,\bm{\lambda})$ satisfies
\begin{equation}
\label{eqz}
D\mathcal{C}(\vh,\alpha,\vecth)(\cdot)= \langle DG(\vh,\alpha,\vecth)(\cdot),\bm{\lambda}\rangle   \qquad \mbox{in $\mathcal H'$.}
\end{equation}
It follows from \eqref{DC} and \eqref{DG} that \eqref{eqz} evaluated in $(\bm{\vec{k}},\beta,\bm{\iota})\in\mathcal H$ is equivalent to
\begin{eqnarray}
\label{eqz2}
\nonumber
\lefteqn{\gamma\sum_{i=1}^n \vec{h}_i\cdot\vec{k}_i+\varepsilon\io \alpha'\beta'+\sum_{i=1}^n\io (\vth_i-\vthb_i)\iota_i = \sum_{i=1}^n\io -\vec{k}_{i}\cdot \lambda_i D\vec m(\alpha+\vth_i)}  \\
&&+\sum_{i=1}^n\io -\lambda_i \vec h_i\cdot D^2\vec m(\alpha+\vth_i)\beta +\sum_{i=1}^n\io \left\{\lambda_i'\iota_i'-\lambda_i \vec h_i\cdot{D^2\vec m}(\alpha+\vth_i) \iota_i\right\}\,.
\end{eqnarray}
Since $D\mathcal{C}$ and $DG$ are linear w.r.t. $(\bm{\vec{k}},\beta,\bm{\iota})$,
\ch{\eqref{eqz2} is equivalent to}
\begin{equation}
\label{eqz3}
\begin{cases}
\displaystyle \io \left\{\lambda_i'\iota_i'-\lambda_i\vec h_i\cdot{D^2\vec m}(\alpha+\vth_i) \iota_i\right\}=\io (\vth_i-\vthb_i) \iota_i, \\
\displaystyle \varepsilon\io \alpha'\beta'+\sum_{i=1}^n\io \lambda_i\vec h_i\cdot D^2\vec m(\alpha+\vth_i)\beta=0, \\
\displaystyle \gamma \vec k_{i}\cdot\vec h_i=-\vec k_i\cdot\io \lambda_i{D\vec m}(\alpha+\vth_i), \\
\end{cases} \forall i=1,\dots,n.
\end{equation}
Recalling Remark \ref{regularity} and adding to \eqref{eqz3} the constraint $(\vh,\alpha,\vecth)\in\mathcal{A}$, we conclude that $(\vh,\alpha,\vecth,\bm{\lambda})$ is a solution to \eqref{EQ}.

\smallskip

In order to deduce $\alpha'(0)=0$ \ch{if $\eps>0$}, we recall that $\vth_i,\alpha,\lambda_i\in C^1([0,1])$ (see Remark \ref{regularity}). Therefore, using $(P_{\alpha})$, \ch{we obtain} $\alpha\in C^2(I)$ and
$$
\begin{aligned}
  \alpha'(1)-\alpha'(0)&=\io \alpha''\stackrel{(P_{\alpha})}=\frac{1}{\varepsilon}\sum_{i=1}^n\io \lambda_i \vec h_i\cdot{D^2\vec m}(\alpha+\vth_i)
   \\ & \stackrel{\eqref{eq:18}}=\frac{1}{\varepsilon}\sum_{i=1}^n\io \vec h_i \cdot \left(\begin{smallmatrix}0 & -1\\[0.5em] +1 & 0\end{smallmatrix}\right)\lambda_i D\vec m(\alpha+\vth_i)
   \stackrel{(P_{\vec h_{i}})}=-\frac \gamma \varepsilon \sum_{i=1}^n\vec h_i \cdot \left(\begin{smallmatrix}0 & -1\\[0.5em] +1 & 0\end{smallmatrix}\right)\vec h_i =0.
  \end{aligned}
$$
Hence $\alpha'(1)=\alpha'(0)\stackrel{(P_{\alpha})}=0$.
\end{proof}

\section{A constructive scheme; uniqueness of solutions to the Lagrange multiplier formulation}\label{sec:uniqueness}

In this section\ch{, where we assume $\eps,\gamma>0$,} we introduce a constructive scheme to obtain solutions of the Euler-Lagrange formulation \eqref{EQ}. We will prove its contractivity and, as a by-product, uniqueness of solutions to \eqref{EQ} (Theorem \ref{exun}). The scheme consists of two steps and works as follows.

\medskip

{\bf Step 1.} In the first step, we fix $\alpha\in C([0,1])$. We introduce the set
\begin{equation}\label{def-D}
D:=\{\vh\in\R^{2n} : \max_{1\le i\le n} |\vec{h}_i|\leq K \}, \quad \mbox{with $K<c_p^{-2}$ (cf. \eqref{eq:15}).}
\end{equation}
We will show that the chain
\begin{equation}\label{scheme}
\vh\stackrel{(P_{\vth_i})}\mapsto\vecth=(\vartheta_1,\dots,\vartheta_n)\stackrel{(P_{\lambda_i})}\mapsto{\bm\lambda}=(\lambda_1,\dots,\lambda_n)\stackrel{(P_{\vec h_i})}\mapsto \bm{\vec T}^{(\alpha)}(\vh)\in\R^{2n}
\end{equation}
defines a map  $\bm{\vec T}^{(\alpha)}:D\to D$. We then show that $\bm{\vec T}^{(\alpha)}$ is a contraction for $\gamma$ sufficiently large. Then, by Proposition \ref{contr}, there exists a unique fixed point of $\bm{\vec{T}}^{(\alpha)}$ in $D$, $\bm{\vec h}(\alpha)$:
$$
\bm{\vec h}(\alpha) = \bm{\vec{T}}^{(\alpha)}(\bm{\vec h}(\alpha)).
$$

{\bf Step 2.} In view of Step 1, we define $A:C([0,1])\to C([0,1])$ as the unique function such that
\begin{equation}
\label{app13}
\begin{cases}
\displaystyle -A(\alpha)''=-\frac{1}{\varepsilon}\sum_{i=1}^n \lambda_i(\alpha)\vec h_i(\alpha)\cdot D^2(\alpha+\vartheta_i(\alpha)) & \mbox{in $(0,1),$} \\
A(\alpha)|_0= A(\alpha)'|_1=0,
\end{cases}
\end{equation}
where $\vartheta_i(\alpha),\lambda_i(\alpha)$ are the unique solutions to \1eq, resp. \2eq, with $\bm{\vec h}=\bm{\vec h}(\alpha)$. We will prove that $A$ is a contraction, hence it has a unique fixed point, for $\gamma$ sufficiently large.

\medskip

Thanks to \eqref{app13} and to \eqref{scheme}, a quadruplet $(\bm{\vec h}(\alpha),\alpha,\bm\vartheta(\alpha),\bm\lambda(\alpha))$ is a solution to System \eqref{EQ} if and only if $\alpha$ is a fixed point of $A$. In particular, this implies the uniqueness result in Theorem \ref{exun}.

\begin{remark}{\rm There are three key features of System \eqref{EQ} that allow us to show that the maps $\vec{\bm T}^{(\alpha)}$ and $A$ are contractions. Namely:
\begin{itemize}
\item the boundary-value problems in \eqref{EQ} share the same structure, that of \eqref{general};
\smallskip
\item all lowest-order terms on the left-hand sides  of the differential equations in \eqref{EQ} are \emph{proportional} to the norm $|\vec h_i|$ of the applied fields, which in turn are controlled (for a minimizer) by the target shapes $\overline{\bm\vth}$ and by the regularization constant $\gamma$ through the bound \eqref{eq:13};
\smallskip
\item the equation for $\vec h_i$ in \eqref{EQ} contain on the right-hand side the pre-factor $\gamma^{-1}$. Accordingly, as long as $\gamma$ is large, we can control the applied fieds and hence also the solutions of \eqref{EQ}.
\end{itemize}
}  \end{remark}

We now prove the assertions formulated above.
\begin{prop}\label{prop-step1+2}
Let  $D$ as in \eqref{def-D}. Then:
\begin{itemize}
\item[$(i)$] For any $\alpha\in C([0,1])$ there exists $\gamma_1$ (depending on $\vthb$ and $K$) such that for any $\gamma>\gamma_1$ the map $\bm{\vec T}^{(\alpha)}:D\to \R^{2n}$ defined in \eqref{scheme} has a unique fixed point in $D$, $\bm{\vec h}(\alpha)$; in particular,
    $$
    \max_{1\le i\le n} |\vec{h}_i(\alpha)|\leq K<c_p^{-2};
    $$
\item[$(ii)$] there exists $\gamma_*> \gamma_1$ (depending on $\vthb$, $K$, and $\varepsilon$) such that the map $A:C([0,1])\to C([0,1])$ defined in \eqref{app13} has a unique fixed point, $\alpha=A(\alpha)$. Furthermore, $\alpha\in C^2([0,1])$ and $\alpha(0)=\alpha'(1)=0$;
\smallskip\item[$(iii)$] consequently, Theorem \ref{exun} holds.
\end{itemize}
\end{prop}

\begin{proof} We divide the proof into steps.

\medskip

{\it (A). There exists $\gamma_0$ such that $\bm{\vec{T}}^{(\alpha)}$ maps $D$ in itself.} Let $\vh\in D$. Thanks to the first equivalence in \eqref{eq:12}, \eqref{trigR2}, and \eqref{eq:15}, we can apply Lemma \ref{qualit} with $L=K$: for every $i=1,\dots,n$ there exists a unique solution $\vth_i\in\hl$ of \1eq, which satisfies
\begin{equation}\label{stima-the}
  \|\vartheta_i\|_\infty \stackrel{\eqref{bound2}_2}\le \ch{|\vec h_i| \le} K, \quad i=1,\dots,n.
\end{equation}
By the same argument, the second equivalence in \eqref{eq:12} and \eqref{trigR} allow to apply Lemma \ref{qualit} with $L= K$ and $f_0=\overline\vth_i-\vth_i$: for every $i=1,\dots,n$ there exists a unique solution $\lambda_i\in \hl$ of \2eq, such that
\begin{eqnarray}
\nonumber
\|\lambda_i\|_\infty & \stackrel{\eqref{bound2}_1}\leq & \frac{c_p}{{1-Kc_p^2}} \|\vartheta_i-\overline\vartheta_i\|_\infty \le \frac{c_p}{{1-Kc_p^2}} \left(\|\vartheta_i\|_\infty + \|\overline\vartheta_i\|_\infty\right)
\\
& \stackrel{\eqref{stima-the}}\le & \frac{c_p}{{1-Kc_p^2}} \left(K + \|\overline\vartheta_i\|_\infty\right)=C,
\label{stima-lam}
\end{eqnarray}
where from now on $C$ denotes a generic constant depending on $\tbar$ and $K$, but independent of $\gamma$ and $\eps$. Therefore
\begin{equation}
\label{app0bis}
|\vec{T}_i^{(\alpha)}(\vh)| \stackrel{\eqref{scheme},\eqref{EQ}_4}= \frac{1}{\gamma}\left|\io \lambda_i {D\vec m}(\alpha+\vth_i)  \right| \stackrel{\eqref{eq:16}}\leq \frac{2}{\gamma} \|\lambda_i\|_\infty \stackrel{\eqref{stima-lam}}\leq \frac{C}{\gamma}, \quad\mbox{$i=1,\dots,n$}.
\end{equation}
This implies that, for $\gamma>\gamma_0$ large enough, the operator $\bm{\vec{T}}^{(\alpha)}$ maps $D$ in itself.

\medskip

For reasons which will be clarified later, we postpone the proof of $(i)$, and for the moment we assume it to be true.

\medskip

{\it (B). Proof of $(ii)$ assuming $(i)$.} Assume $\gamma>\gamma_1$, with $\gamma_1$ as given in $(i)$. For $\alpha$ and $\tilde\alpha$ in $C([0,1])$, let $\vartheta_i=\vartheta_i(\alpha)$ and $\tilde\vartheta_i=\tilde\vartheta_i(\alpha)$, resp. $\lambda_i=\lambda_i(\alpha)$ and $\tilde\lambda_i=\tilde\lambda_i(\alpha)$, be the unique solutions to \1eq, resp. \2eq, with $\bm{\vec h}=\bm{\vec h}(\alpha)$ and $\vec{\tilde{\bm{h}}}=\vec{\tilde{ \bm h}}(\tilde{\alpha})$ as defined in $(i)$.  It follows from \eqref{stima-the} and \eqref{stima-lam} that
\begin{equation}
\label{app0}
\|\vth_i\|_{\infty}+\|\lambda_i\|_{\infty} \le C \quad\mbox{and}\quad  \|\tilde{\vth}_i\|_{\infty}+\|\tilde{\lambda}_i\|_{\infty}\leq C, \quad i=1,\dots,n.
\end{equation}
Let $A(\alpha)$ and $A(\tilde\alpha)$ be defined by \eqref{app13}, and note that \eqref{app13} is equivalent to
\begin{equation}
\label{app13bis}
A(\alpha)(s):=-\frac{1}{\varepsilon}\sum_{i=1}^n\int_0^s\int_{s'}^1 \lambda_i(s'')\vec{h}_i\cdot D^2(\alpha(s'')+\vth_i(s'')) ds'' ds', \qquad \forall s\in[0,1]\,;
\end{equation}
in particular, $A(\alpha)\in C^2([0,1])$. Therefore
\begin{eqnarray}
\label{app14}
\lefteqn{\|A(\alpha)-A(\tilde{\alpha})\|_\infty}
 \nonumber \\
&\stackrel{\eqref{app13bis}} \leq & \frac{1}{\varepsilon}\sum_{i=1}^n \|\lambda_i\vec h_i\cdot D^2\vec m(\alpha+\vth_i)-\tilde{\lambda}_i\vec{\tilde h}_i\cdot D^2\vec m(\tilde{\alpha}+\tilde{\vth}_i)\|_{\infty}
\nonumber \\
&\stackrel{\eqref{trigR}}\leq & \frac{1}{\varepsilon}\sum_{i=1}^n \left(|\vec{h}_i| \|\lambda_i-\tilde\lambda_i\|_{\infty} + \|\tilde{\lambda}_i\|_{\infty} |\vec{h}_i-\vec{\tilde h}_i| + |\vec{\tilde{h}}_i|\|\tilde{\lambda}_{i}\|_{\infty} \left(\|\alpha-\tilde{\alpha}\|_{\infty}+\|\vartheta_i-\tilde{\vartheta}_i\|_{\infty}\right)\right)
\nonumber \\
&\stackrel{\eqref{def-D},\eqref{app0}}\leq &
\frac{C}{\varepsilon}\sum_{i=1}^n \left(|\vec{h}_i-\vec{\tilde h}_i|+\|\alpha-\tilde{\alpha}\|_{\infty}+\|\vartheta_i-\tilde{\vartheta}_i\|_{\infty} +\|\lambda_i-\tilde\lambda_i\|_{\infty}\right).
\end{eqnarray}
Now we will estimate the right hand side of \eqref{app14}. Taking $\lambda_i-\tilde{\lambda}_i$ as test function in the weak formulations for $\lambda_i$ and $\tilde{\lambda}_i$ (cf. \eqref{eq:12} and \eqref{weak-form}) and subtracting the resulting equations, we obtain
\begin{eqnarray}
\label{app19}
\io |(\lambda_i-\tilde{\lambda}_i)'|^2&=&\io \left(\lambda_i\vec h\cdot{D^2\vec m}(\alpha+\vth_i)-\tilde{\lambda}_i\vec{\tilde{h}}_i\cdot D^2\vec m(\tilde\alpha+\tilde{\vth}_i)\right)(\lambda_i-\tilde{\lambda}_i) \nonumber \\
&&+\io(\vth_i-\tilde{\vth}_i)(\lambda_i-\tilde{\lambda}_i) \nonumber 
\\
&\stackrel{\eqref{trigR}}{\le}& \left(\|\tilde{\lambda}_i\|_{\infty}|\vec{\tilde{h}}_i|+1\right)\io |\vth_i-\tilde{\vth}_i||\lambda_i-\tilde{\lambda}_i| \nonumber \\
&&+\|\tilde{\lambda}_i\|_{\infty}|\vec{h}_i-\vec{\tilde{h}}_i|\io |\lambda_i-\tilde{\lambda}_i| +\|\tilde{\lambda}_i\|_{\infty}|\vec{\tilde{h}}_i| \io |\alpha-\tilde{\alpha}||\lambda_i-\tilde{\lambda}_i| \nonumber \\
&&+|\vec{h}_i|\io |\lambda_i-\tilde{\lambda}_i|^2.
\end{eqnarray}
We estimate the last summand on the r.h.s. of \eqref{app19} using the Poincar\'e inequality and the definition of $D$:
$$
|\vec{h}_i|\io |\lambda_i-\tilde{\lambda}_i|^2 \stackrel{\eqref{def-c},\eqref{def-D}}\le K c_p^2 \io |(\lambda_i-\tilde{\lambda}_i)'|^2.
$$
Absorbing this summand on the left-hand side of \eqref{app19}, we obtain:
\begin{eqnarray*}
\lefteqn{\underbrace{(1-Kc_p^2)}_{\text{$> 0$ by }\eqref{eq:15}}\io |(\lambda_i-\tilde{\lambda}_i)'|^2
 \le \left(\io|\lambda_i-\tilde{\lambda}_i|\right) \times}
 \\ && \times \left(
 \left(\|\tilde{\lambda}_i\|_{\infty}|\vec{\tilde{h}}_i|+1\right)\|\vth_i-\tilde{\vth}_i\|_\infty + \|\tilde{\lambda}_i\|_{\infty}|\vec{h}_i-\vec{\tilde{h}}_i| + \|\tilde{\lambda}_i\|_{\infty}|\vec{\tilde{h}}_i|\|\alpha-\tilde{\alpha}\|_\infty\right)
\\ & \stackrel{\eqref{def-D},\eqref{app0}}\leq & C  \left(\io|\lambda_i-\tilde{\lambda}_i|\right) \left(|\vec{h}_i-\vec{\tilde{h}}_i| + \|\vth_i-\tilde{\vth}_i\|_{\infty} + \|\alpha-\tilde{\alpha}\|_\infty\right).
\end{eqnarray*}
Using H\"older and Poincar\'e inequalities, we deduce
\begin{equation}
\label{app20}
\|\lambda_i-\tilde{\lambda}_i\|_{\infty} \stackrel{\eqref{def-S}}\leq \|\lambda_i-\tilde{\lambda}_i\| \leq C\left(|\vec{h}_i-\vec{\tilde{h}}_i| + \|\vth_i-\tilde{\vth}_i\|_{\infty} + \|\alpha-\tilde{\alpha}\|_\infty\right).
\end{equation}
In order to estimate $\|\vth_i-\tilde{\vth}_i\|_{\infty}$, we follow the same line of argument. We choose $\vth_i-\tilde{\vth}_i$ as test function in the weak formulations for $\vth_i$ and for $\tilde{\vth}_i$ (cf. \eqref{weak-form}): subtracting the resulting equations, we have
\begin{eqnarray*}
\io |(\vth_i-\tilde{\vth}_i)'|^2 &=& \io \bigl(\vec h_i\cdot D\vec m(\alpha+\vth_i)-\vec{\tilde{h}}_i\cdot D\vec m(\tilde\alpha+\tilde{\vth}_i)\bigr)(\vth_i-\tilde{\vth}_i) \nonumber
\\
& \stackrel{\eqref{trigR2}}\le & \left(|\vec{h}_i-\vec{\tilde{h}}_i|+|\vec{h}_i|\|\alpha-\tilde\alpha\|_\infty\right)\io|\vth_i-\tilde{\vth}_i|+|\vec{{h}}_i|\io |\vth_i-\tilde{\vth}_i|^2.
\end{eqnarray*}
As above, the second summand may be absorbed on the left-hand side via Poincar\'e inequality and the assumption that $K<c_p^{-2}$, whereas the first one can be treated by H\"older and Poincar\'e inequality (the specific constant being irrelevant in this case). Altogether, we obtain
\begin{equation}
\label{app8}
\|\vth_i-\tilde{\vth}_i\|_{\infty}\stackrel{\eqref{def-S}}\le \|\vth_i-\tilde{\vth}_i\|\leq C\left(|\vec{h}_i-\vec{\tilde{h}}_i|+\|\alpha-\tilde\alpha\|_\infty\right).
\end{equation}

Now we estimate $|\vec{h}_i-\vec{\tilde{h}}_i|$. By the definition of $T_i^{(\alpha)}$ we deduce that
\begin{eqnarray}
\label{app18}
|\vec{h}_i-\vec{\tilde{h}}_i| & \leq &\frac{1}{\gamma}\|\lambda_i{D\vec m}(\alpha+\vth_i)-\tilde{\lambda}_iD\vec m(\tilde{\alpha}+\tilde{\vth}_i)\|_{\infty} \nonumber \\
& \stackrel{\eqref{trigR4},\eqref{app0}}\leq & \frac{C}{\gamma}\left(\|\alpha-\tilde{\alpha}\|_\infty+\|\vth_i-\tilde{\vth}_i\|_\infty+\|\lambda_i-\tilde{\lambda}_i\|_\infty\right).
\end{eqnarray}
Inserting \eqref{app20} and \eqref{app8} in \eqref{app18}, we deduce that there exists $C_2$ such that
\begin{equation}\label{ref-a-dopo}
|\vec{h}_i-\vec{\tilde{h}}_i|\leq\frac{C_2}{\gamma}\left(\|\alpha-\tilde{\alpha}\|_\infty+|\vec{h}_i- \vec{\tilde{h}}_i|\right).
\end{equation}
Taking $\gamma_2=C_2$, we have $C_2/\gamma<1$ for $\gamma>\gamma_2$, so that
\begin{equation}
\label{app22}
|\vec{h}_i-\vec{\tilde{h}}_i|\leq \frac{C_2}{\gamma-C_2}\|\alpha-\tilde{\alpha}\|_\infty
\end{equation}
for $\gamma>\gamma_2$.  Using \eqref{app22} in \eqref{app8}, we obtain
\begin{equation}
\label{app23}
\|\vth_i-\tilde{\vth}_i\|_{\infty}\leq \frac{C}{\gamma-C_2}\|\alpha-\tilde{\alpha}\|_\infty.
\end{equation}
In turn, using \eqref{app22} and \eqref{app23} in \eqref{app20}, we obtain
\begin{equation}
\label{app24}
\|\lambda_i-\tilde{\lambda}_i\|_{\infty}\leq \frac{C}{\gamma-C_2}\|\alpha-\tilde{\alpha}\|_\infty.
\end{equation}
Finally, inserting \eqref{app22}, \eqref{app23}, and \eqref{app24} into \eqref{app14}, we deduce that there exist $C_3$ such that
\begin{equation}
\label{app25}
\|A(\alpha)-A(\tilde{\alpha})\|_\infty\leq \frac 1 \varepsilon \frac{C_3}{\gamma-C_2}\|\alpha-\tilde{\alpha}\|_\infty.
\end{equation}
We now set $\gamma_3=C_2+C_3/\varepsilon >\gamma_2$, so that the prefactor in \eqref{app25} is smaller than $1$ for every $\gamma>
\max\{\gamma_1,\gamma_3\}=:\gamma_*$.  By Proposition \ref{contr}, for $\gamma>\gamma_*$ there exists a unique fixed point $\alpha\in C([0,1])$ of $A$.

\medskip

We finally return to the proof of $(i)$, which we postponed since its proof is simpler than that of $(ii)$, in that we may use the same estimates as in $(B)$ with $\alpha=\tilde\alpha$.

\medskip

{\it (C). Proof of (i).} We prove that $\bm{\vec{T}}^{(\alpha)}$ is a contraction. Let $\gamma>\gamma_0$, as given in $(A)$. Given $\vh, \vht\in D$, we define  $\vth_i$ and $\tilde{\vth}_i$, resp. $\lambda_i$ and  $\tilde{\lambda}_i$, as the
corresponding unique solutions of \1eq, resp. \2eq. Then, the same arguments of $(B)$ may be applied with $\alpha=\tilde\alpha$, yielding
\begin{equation}
\label{app5}
\|\lambda_i-\tilde{\lambda}_i\|_{\infty} \leq C\left(|\vec{h}_i-\vec{\tilde{h}}_i|+\|\vth_i-\tilde{\vth}_i\|_{\infty}\right)
\end{equation}
(cf. \eqref{app20}) and
\begin{equation}
\label{app8bis}
\|\vth_i-\tilde{\vth}_i\|_{\infty}\stackrel{\eqref{def-S}}\le \|\vth_i-\tilde{\vth}_i\|\leq C|\vec{h}_i-\vec{\tilde{h}}_i|
\end{equation}
(cf. \eqref{app8}). Therefore
\begin{eqnarray*}
|\vec{T}_i^{(\alpha)}(\vh)-\vec{T}_i^{(\alpha)}(\vht)|
&\stackrel{\eqref{scheme},\eqref{EQ}_4}= &
\frac{1}{\gamma}\left| \io \left(\lambda_i {D\vec m}(\alpha+\vth_i) -\tilde\lambda_i {D\vec m}(\alpha+\tilde{\vth}_i)\right)\right|
\\ & \stackrel{\eqref{trigR4},\eqref{stima-lam}}\leq & \frac{C}{\gamma}\left\{\|\lambda_i-\tilde{\lambda}_i\|_{\infty}+\|\vth_i-\tilde{\vth}_i\|_{\infty}\right\}
 \stackrel{\eqref{app5},\eqref{app8bis}}\leq \frac{C_1}\gamma  |\vec{h}_i-\vec{\tilde{h}}_i|.
\end{eqnarray*}
Choosing $\gamma_1=C_1$, we conclude that $\bm{\vec{T}}^{(\alpha)}$ is a contraction for every $\gamma>\gamma_1$.

\medskip

{\it (D). Proof of (iii).} Theorem \ref{exun} is an immediate consequence of $(i)$ and $(ii)$. Indeed, let $\alpha=A(\alpha)$ be the fixed point of $A$ identified in $(ii)$, and let $\vh(\alpha)=\bm{\vec{T}}^{(\alpha)}(\vh(\alpha))$ be the fixed point identified in $(i)$.  Then, by construction, the quadruplet $(\vh(\alpha),\alpha, \bm\vartheta(\alpha), \bm\lambda(\alpha))$ is a solution to system \eqref{EQ} in the class \eqref{cond}. Viceversa, if two solutions of \eqref{EQ} exist in that class, then they are both fixed points of $A$, hence they coincide.

\end{proof}

\begin{remark}{\rm
 Under the provision that $K<1$, a bound similar to \eqref{app5} might be obtained from the representation formula
  \begin{equation*}
\lambda_i-\tilde\lambda_i=\int_0^s\int_{s'}^{1}\left(\lambda_i\vec h_i\cdot D^2\vec m(\alpha+\vth_i)-\tilde\lambda_i\vec{\tilde{ h}}_i\cdot D^2\vec m(\alpha+\tilde{\vth}_i)\right)ds''ds'.
\end{equation*}
Indeed, from
\[
    \|\lambda_i-\tilde\lambda_i\|_{\infty}\stackrel{\eqref{trigR}}\le \|\tilde{\lambda_i}\|_{\infty}|\vec{\tilde{ h}}_i|\|\vth_i-\tilde{\vth}_i\|_{\infty}+\|\tilde\lambda_i\|_{\infty}|\vec h_i-\vec{\tilde{h}}_i|+|\vec h_i|\|\lambda_i-\tilde\lambda_i\|_{\infty},
  \]
we obtain, for $K<1$,
    \begin{eqnarray*}
      \|\lambda_i-\tilde\lambda_i\|_{\infty} &  \stackrel{\eqref{stima-lam},\eqref{def-D}} \leq & \frac 1{1-K}\left( K C(\overline{\bm\vth},K)\|\vth_i-\tilde{\vth}_i\|_{\infty}+C(\overline{\bm\vth},K)|\vec h_i-\vec{\tilde h}_i|\right)\\
      & \leq  &  C(\overline{\bm\vth},K)\bigl(\|\vth_i-\tilde{\vth}_i\|_{\infty}+|\vec h_i-\vec{\tilde h}_i|\bigr).
    \end{eqnarray*}
However, the requirement $K<1$ is stricter than our assumption \eqref{eq:15} because $c_p<1$.
}\end{remark}

\begin{remark}\label{rem:technical}
\ch{\rm Note that $\gamma_{**}=\max(\overline\Theta^2/K^2,\gamma_*)$ blows up both as $K$ tends to $0$ and as $K$ tends to $c_p^{-2}$. The blow up for $K$ small is obvious: as $K$ tends to $0$ the maximum allowed applied field tends to $0$ in intensity, and to limit the applied field we need  $\gamma$ large. The blow up as $K\to c_p^{-2}$ is of a technical nature and follows from the blow up of $\gamma_*$. In turn, the blow up of $\gamma_*$ follows from the estimates in the proof of Theorem \ref{exun}, which become degenerate as $K$ tends to $c_p^{-2}$ (see e.g. \eqref{stima-lam}-\eqref{app0bis}). Ultimately, this is because our estimates rely only on \eqref{bound2}$_1$ in Lemma \ref{qualit}, which becomes degenerate when the Lipschitz constant $L$ (identified with the intensity of the magnetic field) approaches $c_p^{-2}$.
}
\end{remark}

Let us conclude the Section with a digression on the case in which $\alpha$ is fixed. Part $(i)$ of Proposition \ref{prop-step1+2} may be rephrased as follows:
\begin{prop}
\label{lamu13}
Let $\tbar\in C([0,1])^n$ and let $K>0$ such that \eqref{eq:15} holds. Then there exists $\gamma_*=\gamma_*(\tbar,K)$ such that for every $\gamma> \gamma_*$ and any $\alpha\in C([0,1])$ there exists a unique solution of the system
\begin{equation}
\label{EQ1}
\left\{\begin{array}{ll}
\displaystyle -\vth_i''-\vec h_i\cdot{D\vec m}(\alpha+\vth_i)=0, & \vartheta_i(0)=\vartheta'_i(1)=0 \\
\displaystyle -\lambda_i''-\lambda_i \vec h_i\cdot{D^2\vec m}(\alpha+\vth_i)=\vth_i-\vthb_i, & \lambda_i(0)=\lambda_i'(1)=0 \\
\displaystyle \vec{h}_{i}=-\frac{1}{\gamma}\io \lambda_i {D\vec m}(\alpha+\vth_i) &\\
\end{array}\right.
\forall i=1\dots,n
\end{equation}
within the following set:
$(\vh,\vecth,\bm{\lambda})\in D\times\hl^n \times\hl^n$.
\end{prop}

\begin{remark}{\rm
\label{minser}
For fixed $\alpha$, the solution in Proposition \ref{lamu13} is the unique stationary point of the functional
\begin{equation}
\label{app26}
\tilde{\mathcal{C}}(\vh,\vecth)=\frac{1}{2}\sum_{i=1}^n \io |\vth_i-\vthb_i|^2 + \frac{\gamma}{2} \sum_{i=1}^n |\vec{h}_i|^2,
\end{equation}
in the admissible set
$$
(\vh,\vecth)\in\tilde{\mathcal{A}}:=\left\{(\vh,\vecth)\in D\times\hl^n: \ \vth_i \text{ solves \eqref{pthi} for every } i=1,\dots,n \right\}.
$$
Therefore, arguing as we did for the full problem, for $\gamma$ sufficiently large it follows from Proposition \ref{lamu13} that:
\begin{itemize}
\item $\tilde{\mathcal{C}}$ has a unique minimizer;

\smallskip

\item looking for the minimum $(\vh,\vecth)$ of $\tilde{\mathcal{C}}$ is equivalent to looking for the fixed point of $\bm{\vec{T}}^{(\alpha)}$.
\end{itemize}
}\end{remark}

\section{\ch{Concluding remarks}}\label{sec:concl}
We have considered a beam clamped at one side, modeled as a planar elastica. The beam has a permanent magnetization (the design $\alpha$), hence it deforms under the action of spatially-constant magnetic fields $\vec h_i$, $i=1,\dots,n$ (the controls). Given a list of $n$ prescribed target shapes $\overline \vartheta_i$ ($i=1,\dots,n$), we have looked for optimal design and controls in order for the corresponding shapes $\vartheta_i$ ($i=1,\dots,n$) of the beam to get as close as possible to the corresponding targets. Choosing the cost functional as in \eqref{cost} has lead us to the formulation of an optimal design-control problem (cf. \eqref{main-problem}), whose minimization has been studied by both direct and indirect methods. Loosely speaking, we have shown that:
\begin{itemize}
\item minimizers \ch{$(\vh,\alpha)$} exist (\ch{Theorem} \ref{min});
\smallskip
\item provided the intensity of $\vh$ is sufficiently small (cf. \eqref{bound-on-H})\ch{, and otherwise in ``generic'' cases (see Theorem \ref{lagmul2} and the comments below it)}, minimizers solve the Lagrange multiplier formulation \eqref{EQ};
\smallskip
\item if the parameter $\gamma$ penalizing the cost of the fields' intensity is sufficiently \ch{large}, the minimizer is unique, satisfies \eqref{bound-on-H}, and is the unique solution to the Lagrange multiplier formulation \eqref{EQ} (Theorem \ref{exun} and \ch{Theorem} \ref{coro-un}).
\end{itemize}

In what follows, we briefly discuss a numerical scheme which naturally emerges from the proof of Theorem \ref{exun}, as well as a different choice of the cost functional, using residuals.
\ch{
We also point out two possible generalizations of our choice of the cost.
}

\begin{remark}[\bf The numerical scheme]
{\rm  The proof of Theorem \ref{exun} suggests an alternative to the numerical scheme proposed in \cite{ciambella2019form}. The new
  scheme is based on two nested
  \begin{minipage}{0.5\linewidth}loops. In the inner loop, $\alpha$ is fixed and $\vec{\bm h}$, $\bm\lambda$, and $\bm\vartheta$ are computed by a fixed point iteration scheme which uses, in the order, equations $(P_{\vartheta_i})$, $(P_{\lambda_i})$ and $(P_{\vec h_i})$; in the outer loop, $\alpha$ is updated by using the equation $(P_{\alpha})$ with $\vec{\bm h}$, $\bm\lambda$, and $\bm\vartheta$ obtained from the inner loop. Each loop terminates when the update of each variable results in an increment below a certain tolerance $tol$. The algorithm is described in the pseudocode aside. Note that, in this algorithm, steps to be performed for $i=1,\dots n$ do not need to be carried out sequentially, but can also be done in parallel, since they are independent on each other.
  \end{minipage}\quad
\begin{minipage}[H]{0.5\textwidth}
\tiny  \begin{algorithm}[H]
    {\bf Initialisation}:\\
    \qquad $\alpha\longleftarrow \text{initial guess }\alpha^{(0)}$\;
    \qquad $\vec h_i\longleftarrow\text{initial guess }\vec h_i^{(0)}$, $i=1,\ldots n$\;
    \qquad $\lambda_i\longleftarrow\text{initial guess }\lambda_i^{(0)}$, $i=1,\ldots n$\;
    \qquad $tol\longleftarrow\text{tolerance}$\;
    \Repeat{$\|\alpha^{old}-\alpha\|_{\infty}\le tol$}
    {
    \Repeat{$\max_{i=1}^n |\vec h_i^{old}-\vec h_i|\le tol$}
    {
    {$\vartheta_i\longleftarrow$ solve $(P_{\vartheta_i})$, $i=1,\ldots, n$\;
      $\lambda_i\longleftarrow$ solve $(P_{\lambda_i})$, $i=1,\ldots, n$\;
      $\vec h_i^{old}\longleftarrow \vec h_i$\;
      $\vec h_i\longleftarrow $ solve $(P_{\vec h_i})$, $i=1,\ldots, n$\;
    }
  }
   $\alpha^{old}\longleftarrow\alpha$\;
    $\alpha\longleftarrow$ solve $(P_{\alpha})$\;
}
\end{algorithm}
\end{minipage}
\par
  }
\end{remark}

\begin{remark}[\bf Using residuals to assess shape attainment]\label{rem-residuals}
{\rm
  Shape programming has been addressed in \cite{lum2016shape} under slightly more general conditions than those considered in this paper. In particular, \cite{lum2016shape} allows the  magnetization intensity to be non-constant and the magnetic field to be non-uniform, and assigns a different weight to each shape. Within our framework (constant magnetic intensity, uniform applied field, and same weight for all shapes), the approach proposed in \cite{lum2016shape} would lead to the minimization of the following functional:
 \begin{equation}
\label{costilde}
\widetilde E\ch{(\vh,\alpha)}=\sum_{i=1}^n \int_0^1\left| -\overline\vartheta_i''-\vec h_i\cdot\ch{D}\vec m(\alpha+\overline\vth_i)\right|^2.
\end{equation}
The integrands in \eqref{costilde} represent {\it residuals}, in the sense that they vanish \ch{on attainable targets.} Such minimization would be carried out in the space of designs $\alpha$ whose first $k$ Fourier coefficients are in a bounded set and control fields $\vec{\bm h}$ whose magnitude does not exceed a constant $K$. \ch{It would be useful} to have estimates of the \ch{\emph{attainment error}:}
\[
        E\ch{(\vh,\alpha)}=\frac 12 \sum_{i=1}^n \int_0^1 |\overline\vartheta_i-\Theta_\alpha(\vec h_i)|^2
      \]
      (cf. \eqref{def:Theta_alpha}) for solutions of both the optimization problem considered in \cite{lum2016shape} and the problem considered in this paper.
      In this respect, a first problem to be solved would be obtaining a bound of $E\ch{(\vh,\alpha)}$ in terms of $\widetilde E\ch{(\vh,\alpha)}$, where $\ch{(\vh,\alpha)}$ is a minimizer of \eqref{costilde}.
}
\end{remark}

\begin{remark}[\bf Variable intensity of the magnetization] {\rm Further developments of the present work may include a variable intensity of the magnetization. In this case, if we let $\mu(s) M_0$  be the magnetization density in the undeformed configuration, then the energy functional \ch{\eqref{eq:6}} would be replaced by
    \[
      \widetilde{\mathcal E}(\vartheta)=\int_0^1\frac 12 (\vartheta')^2-\mu \vec h\cdot\vec m(\vartheta+\alpha).
    \]
Such modification would also require a regularization to limit the oscillations of $\mu$, as well as a penalization of negative values. Instead of choosing $\mu$ and $\alpha$ as design variables for the magnetization, one might choose the vector $\vec\mu=\mu\vec m(\alpha)$. In terms of this vector, the energy would take the form
 \[
      \widetilde{\mathcal E}(\vartheta)=\int_0^1\frac 12 (\vartheta')^2-\vec h\cdot\mathbf R(\vartheta)\vec\mu,
    \]
    where $\mathbf R(v)=\left(\begin{smallmatrix}\cos v & -\sin v \\[0.5em] \sin v & \cos v\end{smallmatrix}\right)$ is the counterclockwise rotation of the angle $v$. Such extension should be accompanied by a penalization of the oscillation of the vector field $\vec\mu$.}
\end{remark}

\begin{remark}[\bf Non-quadratic costs] \ch{{\rm A non trivial generalization of the present work consists in considering more general costs, of the form
 \begin{equation}
\label{pcost}
\sum_{i=1}^n \|\overline\vartheta_i-\vartheta\|_{L^2}^{2} +e(\alpha,\alpha')+ g(\vh).
\end{equation}
Of particular interest might be obstacle-type penalization. Such more general situation  would likely require different techniques, with respect to those used in this paper.
}}
\end{remark}

\section*{Appendix}
\setcounter{defin}{0}
\setcounter{equation}{0}
\renewcommand{\theequation}{A.\arabic{equation}}
\renewcommand{\thedefin}{A.\arabic{defin}}

\subsection*{Reminders of calculus in Banach spaces} Let $(\mathcal{X},\|\cdot\|_{\mathcal{X}}), (\mathcal{Y},\|\cdot\|_{\mathcal{Y}})$ be Banach spaces, $\mathcal{U}$ an open subset of $\mathcal{X}$. We shall also consider a generic map $F:\mathcal{U}\to\mathcal{Y}$.
\begin{defin}
\label{gat}
$F$ has G\^ateaux derivative $F'(x_0)$ at the point $x_0\in\mathcal{U}$ if there exist
$$
\displaystyle F'(x_0)(v):=\lim_{t\to 0} \frac{F(x_0+t v)-F(x_0)}{t}, \qquad \forall v\in\mathcal{X}.
$$
\end{defin}
\begin{defin}
\label{fre}
$F$ is called Fr\'echet differentiable at $x_0\in\mathcal{U}$ if there exists $DF(x_0)\in\mathcal{L}(\mathcal{X},\mathcal{Y})$ such that
$$
\displaystyle \lim_{\| h\|_{\mathcal{X}}\to 0}\frac{\|F(x_0+h)-F(x)-DF(x_0)(h)\|_{\mathcal{Y}}}{\|h\|_{\mathcal{X}}}=0.
$$
\end{defin}
Moreover we give the notion of continuous differentiable operator. Let $T$ belong to $\mathcal{L}(\mathcal{X},\mathcal{Y})$. We recall that the operator norm is defined by
$$
\|T\|_{\mathcal{L}(\mathcal{X},\mathcal{Y})}:=\sup_{\{0\ne x\in\mathcal{X}\}} \frac{\|T(x)\|_{\mathcal{Y}}}{\|x\|_{\mathcal{X}}}=\sup_{\|x\|_{\mathcal{X}}=1} \|T(x)\|_{\mathcal{Y}}.
$$
\begin{defin}
\label{c1}
We say that $F$ is $C^1$ if $DF(x)$ exists for every $x\in\mathcal{U}$ and $DF:\mathcal{U}\to\mathcal{L}(\mathcal{X},\mathcal{Y})$ is a continuous operator.
\end{defin}
We recall the following proposition linking G\^ateaux derivability and Fr\'echet differentiability.
\begin{prop}
\label{equiv}
If $F$ admits G\^ateaux derivative $F'(x)$ in an open neighborhood $\mathcal{V}\subset\mathcal{U}$ of $x_0$ and $F':\mathcal{V}\to\mathcal{L}(\mathcal{X},\mathcal{Y})$ is continuous at $x_0$, then $F$ is Fr\'echet differentiable at $x_0$ and $DF(x_0)=F'(x_0)$. Moreover if $F':\mathcal{U}\to\mathcal{L}(\mathcal{X},\mathcal{Y})$ is a continuous operator, then $DF= F'$ and $F$ is $C^1$.
\end{prop}
\begin{proof}
  See \cite{Zeidler1995}, p. 274.
\end{proof}
We denote by $\mathcal{X}'$ the dual space of $\mathcal{X}$ and by $\langle\cdot,\cdot\rangle:\mathcal{X}'\times\mathcal{X}\to\R$ the duality pairing defined as $\langle S ,t \rangle=S(t)$, for every $t\in\mathcal{X}, S\in\mathcal{X}'$.
\begin{prop}[Existence of a Lagrange multiplier: \cite{Zeidler1995}, p. 270]
\label{lag}
Let $f:\mathcal{U}\subset\mathcal{X}\to\R$ and $G:\mathcal{U}\subset\mathcal{X}\to\mathcal{Y}$ be $C^1$ on an open neighborhood $\mathcal U$ of $\tilde{x}$. Suppose that $\tilde{x}$ is an extremum of $f$ on the set $\{x\in\mathcal{U} : G(x)=0\}$ and that
$$
DG(\tilde{x}):\mathcal{X}\to\mathcal{Y} \qquad \text{is a surjective linear operator}.
$$
Then there exists $\lambda\in\mathcal{Y}'$, a \emph{Lagrange multiplier}, such that
$
Df(\tilde{x})-\langle\lambda,DG(\tilde{x})\rangle=0.
$
\end{prop}

\begin{prop}[{Contraction Theorem \cite[p.18]{Zeidler1995a}}]
\label{contr}
Let $T:\mathcal{X}\to\mathcal{X}$. If $L\in(0,1)$ exists such that
$$
\|T(x)-T(y)\|_{\mathcal{X}}\leq L\|x-y\|_{\mathcal{X}} \qquad\ch{\mbox{for all }}\ x,y\in\mathcal{X},
$$
then $T$ admits a unique fixed point $x^*\in\mathcal{X}$ (i.e. $T(x^*)=x^*$).
\end{prop}

\ch{
\subsection*{Surjectivity of Sturm-Liouville operators}

We prove a result on surjectivity of Sturm-Liouville operators onto dual spaces, which is crucial in the proof of Theorem \ref{lagmul2} and for which we could not find a reference. For $\aut \in \R$, let $a_\aut: H^{1}_{0L}(I)\times H^{1}_{0L}(I)$ be the bilinear symmetric form defined by
\begin{equation}
\label{defa}
a_\aut(u,v)=\int_0^1 u'v' + \int_0^1 quv - \aut\int_0^1 ruv,
\end{equation}
where $q,r:I\to \R$ with $r>0$ in $I$ and $q,r\in L^1(I)$. Let $L_\aut$ be the associated linear operator:
\begin{equation}
\label{defL}
L_\aut:H^{1}_{0L}(I)\to H^{1}_{0L}(I)', \qquad L_\aut(u)(v) =a_\aut(u,v) \qquad \forall v\in H^1_{0L}(I).
\end{equation}
We note that $L_\aut(u)=0$ in $H^{1}_{0L}(I)'$ if and only if $u$ is a weak solution to
\begin{equation}
\label{eqSL}
\begin{cases}
-u''+qu=\aut ru & \mbox{in $(0,1),$} \\
u(0)=0, \ u'(1)=0.
\end{cases}
\end{equation}
We introduce the weighted scalar product $(u,v)_r= \int_0^1 ruv$, with corresponding norm $\|u\|^2_r=(u,u)_r$, and define $L^2(I,r):=\left\{u\in L^1(I) : \ \|u\|_r < +\infty\right\}$. Following the standard nomenclature in Sturm-Liouville theory, we say that:
\begin{defin}\label{def-eigen}
$\aut\in \R$ is an {\sl eigenvalue} of $L_\aut$ if there exists a nonzero function $\varphi\in H^1_{0L}(I)$, called {\sl eigenfunction}, such that $\varphi$ is a solution of $L_\aut(\varphi)=0$.
\end{defin}
The eigenvalues of $L_\aut$ are characterized as follows:
\begin{theorem}[Theorem 4.6.2 of \cite{Zettl}]
\label{baseSL}
Let $L_\aut$ be as in \eqref{defa}-\eqref{defL}. Then
\begin{itemize}
\item[$(i)$] all eigenvalues of $L_\aut$ are real and simple;
\item[$(ii)$] the eigenvalues of $L_\aut$  are an infinite and countable set, $\{\aut_k : k\in\N\}$, bounded from below:
$-\infty < \aut_1 < \aut_2 < \aut_3 < \dots$ and $\aut_k \to \infty$ as $k \to \infty$;
\item[$(iii)$] the sequence of eigenfunction $\{\varphi_k\}$ is complete in $L^2(I,r)$ and can be normalized to be an orthonormal sequence in $L^2(I,r)$.
\end{itemize}
\end{theorem}
\begin{remark}
\label{baseSL1}
It is readily checked that:
\begin{itemize}
\item[$(i)$] If $r\in L^\infty(I)$ and $\inf r>0$, then $L^2(I,r)=L^2(I)$;

\item[$(ii)$] If $q\geq 0$, then $\aut_k> 0$ for every $k\in\N$.
\end{itemize}
\end{remark}

\begin{lemma}
\label{char}
Let $L_\aut$ be as in \eqref{defa}-\eqref{defL} and let $q\geq 0$ and $r\in L^\infty(I)$ be such that $\inf r>0$. Then for every $k\in\N$
\begin{equation}\label{char-eq}
\aut_k = \min\left\{J(v):=\int_0^1 ((v')^2 + qv^2) \ : \ v\in E^\perp_{k-1}, \ \|v\|_r=1\right\},
\end{equation}
where $E^\perp_0=H^1_{0L}(I)$ and $E^\perp_{k}=\left\{ v\in H^1_{0L}(I) : (\varphi_i ,v)_r =0 \quad \forall i=1,\dots,k\right\}$ for $k\geq 1$.
\end{lemma}

\begin{proof}
Let $m_k=\inf \left\{J(v): \ v\in E_{k-1}^\perp , \  \|v\|_r=1 \right\}$. Since $J$ is weakly lower semicontinuous and coercive and $\left\{v\in E_{k-1}^\perp : \|v\|_r=1 \right\}$ is weakly closed in $H^1_{0L}(I)$, there exists a minimizer $u_k\in  E_{k-1}^\perp$, $m_k=J(u_k)$, and $\|u_k\|_r=1$. Furthermore, $u_k$ solves
\begin{equation}
\label{rabb12}
\int_0^1 u_k' v' +\int_0^1 qu_k v = m_k\int_0^1 r u_k v \qquad \forall v\in E_{k-1}^\perp.
\end{equation}
We now prove that $m_k=\aut_k$. In view of Remark \ref{baseSL1} $(i)$, we can decompose any $v\in H^1_{0L}(I)$ as $v=\sum\limits_{i=1}^{k-1}v_i \varphi_i + w$ with $w\in E_{k-1}^\perp$.  Noting that
$$
\int_0^1 u_k' \varphi_i'+\int_0^1 qu_k\varphi_i=\aut_i \int_0^1 ru_k\varphi_i \stackrel{u_k\in E_{k-1}^\perp}= 0 \quad\mbox{for $i=1,\dots,k-1$},
$$
we deduce that \eqref{rabb12} holds for any $v\in H^1_{0L}(I)$. This implies that $m_k$ is an eigenvalue, and, since $u_k\in E_{k-1}^\perp$, $m_k\geq \aut_k$. On the other hand $m_k=J(u_k)\leq J(\varphi_k)=\aut_k$, hence we conclude that $m_k=\aut_k$.
\end{proof}

\begin{prop}
\label{surSL}
Let $L_\aut$ be as in \eqref{defa}-\eqref{defL}. Let $q\geq 0$ and $r\in L^\infty(I)$ be such that $\inf r>0$. If $\aut$ is not an eigenvalue of $L_\aut$, then $L_\aut$ is surjective.
\end{prop}

\begin{proof}
Our goal is to prove that for every $T\in H^1_{0L}(I)'$ there exists $u\in H^{1}_{0L}(I)$ such that $L_\aut(u)(v)=\langle T, v\rangle$ for all $v\in H^{1}_{0L}(I)$.
Since $\aut\not\in\{\aut_k\}$, we may distinguish two cases: either $\aut < \aut_1$ $(a)$, or $\aut_k<\aut<\aut_{k+1}$ for some $k\ge 1$ $(b)$.
We first analyze case $(b)$.

\medskip

{\it (B). Proof under assumption $(b)$.}
Thanks to Theorem \ref{baseSL} $(iii)$ and Remark \ref{baseSL1} $(i)$, if $u,v\in H^1_{0L}(I)$ then there exist coefficients $u_1,\dots,u_k$ and $v_1,\dots,v_k$ such that
\begin{equation}
\label{rabb0}
u=\sum_{i=1}^k u_i\varphi_i + \overline u \quad \text{and}\quad v= \sum_{i=1}^k v_i\varphi_i + \overline v,
\end{equation}
where $\overline{u},\overline{v}\in E^{\perp}_k$. Using the bilinearity of $a$, we have
\begin{equation}
\label{rabb1}
L_\aut(u)(v) \stackrel{\eqref{rabb0}}=\sum_{i,j=1}^{k}u_i v_j L_\aut(\varphi_i)(\varphi_j) + \sum_{i=1}^{k}u_i L_\aut(\varphi_i)(\overline v) + \sum_{j=1}^{k} v_j L_\aut(\overline u)(\varphi_j) + L_\aut(\overline u)(\overline v).
\end{equation}
Since $\varphi_i$ is an eigenfunction, we deduce that
$$
L_\aut(\varphi_i)(\overline v) \stackrel{\eqref{defL}} = \int_0^1 \varphi_i'\overline v' + \int_0^1 q\varphi_i \overline v - \aut\int_0^1 r\varphi_i \overline v
 =  (\aut_i-\aut)\int_0^1 r\varphi_i\overline v \ \stackrel{\overline v \in E^\perp_k}= \ 0 
$$
for $i=1,\dots,k$. Analogously, $L_\aut(\overline u)(\varphi_j)=0$ for $j=1,\dots,k$. Hence \eqref{rabb1} turns into
\begin{equation}
\label{rabb2}
L_\aut(u)(v)=\sum_{i,j=1}^{k}u_i v_j L_\aut(\varphi_i)(\varphi_j) + L_\aut(\overline u)(\overline v).
\end{equation}
By definition of $L_\aut$ and using first that $\varphi_i$ is an eigenfunction with eigenvalue $\aut_i$ and then Theorem \ref{baseSL} $(iii)$, we have
\begin{eqnarray}
\label{rabb3}
L_\aut(\varphi_i)(\varphi_j) &=& \int_0^1 \varphi_i'\varphi_j' + \int_0^1 q\varphi_i \varphi_j - \aut\int_0^1 r\varphi_i \varphi_j \nonumber \\
& = & (\aut_i-\aut)\int_0^1 r\varphi_i\varphi_j  =  \begin{cases} (\aut_i-\aut) & \text{if $i=j$} \\
                    0 & \text{otherwise}.
      \end{cases}
\end{eqnarray}
It follows from \eqref{rabb2} and \eqref{rabb3} that $L_\aut$ is surjective if and only for every $T\in H^1_{0L}(I)'$ there exist coefficients $u_1,\dots,u_k$ and $\overline u\in E_k^\perp$ such that
\begin{equation}
\label{goal2}
\sum_{i=1}^k u_iv_i(\aut_i-\aut) + L_\aut(\overline u)(\overline v) = \langle T, v\rangle \qquad \forall v\in H^{1}_{0L}(I).
\end{equation}
The coefficients $u_1,\dots,u_k$ are readily identified: choosing $v=\varphi_i$ in \eqref{goal2}, since $L_\aut(\overline u)(\varphi_i)=0$ and $\aut\notin\{\aut_i\}$ we deduce
\begin{equation}
\label{rabb6}
u_i=\frac{1}{\aut_i-\aut} \langle T, \varphi_i\rangle  \qquad\forall i=1,\dots,k.
\end{equation}
Assume for a moment that
\begin{equation}
\label{rabb10}
\exists \ \overline u\in E^\perp_k: \qquad L_\aut(\overline u)(\overline v) \stackrel{\eqref{defL}}= a_\mu(\overline u,\overline v) = \langle T, \overline v\rangle  \quad \forall \overline v\in E^\perp_k.
\end{equation}
Then $u$ defined as in \eqref{rabb0} satisfies \eqref{goal2}. Indeed, plugging \eqref{rabb6} and \eqref{rabb10} into the left-hand side of \eqref{goal2} we obtain
\begin{eqnarray*}
L_\aut(u)(v) & =  & \sum_{i=1}^k v_i\langle T, \varphi_i\rangle +  \langle T, \overline v\rangle  =  \langle T, \sum_{i=1}^k v_i\varphi_i + \overline v \rangle =  \langle T, v\rangle.
\end{eqnarray*}
Therefore it remains to prove \eqref{rabb10}.

\medskip

{\it Proof of \eqref{rabb10}.} We note that $E^\perp_k\subset H^1_{0L}(I)$ is a Hilbert space. Therefore $T\in H^1_{0L}(I)'\subset (E^\perp_k)'$. Hence \eqref{rabb10} follows from Lax-Milgram theorem, once we show that the bilinear form $a$, restricted to $E^\perp_k\times E^\perp_k$, is bounded and coercive. Applying H\"older inequality, \eqref{def-S}, and \eqref{def-c}, we have
\begin{eqnarray*}
|a(\overline u,\overline v)| &=& \left|\int_0^1 \overline u'\overline v' + \int_0^1 q\overline u \overline v - \aut\int_0^1 r\overline u \overline v \right|\\
& \leq &  \|\overline u\|\|\overline v\| + \|q\|_{L^{1}(I)}\|\overline u\|_{\infty}\|\overline v\|_{\infty} + \aut c_p^2\|r\|_{\infty} \|\overline u\| \|\overline v\| \le C \|\overline u\|\|\overline v\|,
\end{eqnarray*}
hence $a$ is bounded. In order to prove that $a$ is coercive, we recall that by $(ii)$ of Remark \ref{baseSL1}, $0<\mu_k<\mu<\mu_{k+1}$, and we estimate
\begin{eqnarray}
\nonumber
\lefteqn{a(\overline u,\overline u)  = \int_0^1 \left((\overline u')^2 + q\overline u^2 - \aut r\overline u^2\right)} \\
 & = & \left(1-\frac{\aut}{\aut_{k+1}}\right)\int_0^1 \left((\overline u')^2 + q\overline u^2\right)
 +  \frac{\aut}{\aut_{k+1}}\int_0^1 \left((\overline u')^2+ q\overline u^2 - \aut_{k+1} r\overline u^2\right).\label{rabb8}
\end{eqnarray}
Since $\overline u\in E^\perp_k$, \eqref{char-eq} implies that the second summand on the right-hand side of \eqref{rabb8} is nonnegative; hence
\begin{equation*}
a(\overline u,\overline u) \geq \left(1-\frac{\aut}{\aut_{k+1}}\right)\left(\int_0^1 (\overline u')^2 + \int_0^1 q\overline u^2\right) \stackrel{q\geq 0}\geq \left(1-\frac{\aut}{\aut_{k+1}}\right)\|\overline u\|_{E^\perp_k}^2.
\end{equation*}
Since $\aut<\aut_{k+1}$, $a$ is coercive.

\medskip

{\it (A). Proof under assumption $(a)$.} The result follows by arguing in the same way used in the proof of \eqref{rabb10} in $(B)$, choosing $k=0$ and recalling that $E_0^\perp=H^1_{0L}(I)$.
\end{proof}

In the body of this manuscript, we deal with linear operators $\tilde L: H^1_{0L}(I)\to H^1_{0L}(I)'$ defined as
\begin{equation}
\label{defL2}
\tilde L(u)(v)=\int_0^1 u'v' - \int_0^1 \tilde ru v \qquad \forall v\in H^1_{0L}(I),
\end{equation}
with $\tilde r\in L^\infty(I)$ which may not be positive. However, it turns out that Proposition \ref{surSL} may be applied. Indeed, we have $\tilde L=\tilde L_{1}$, where
\begin{equation}
\label{defL2bis}
\tilde L_{\aut}(u)(v) = \int_0^1 u'v' + \int_0^1 (\tilde r^-+\delta) uv - \aut\int_0^1 (\tilde r^++\delta)uv \qquad \forall v\in H^1_{0L}(I),
\end{equation}
with $r^-=\max(-r,0)$, $r^+=\max(r,0)$ and $\delta >0$ an arbitrary constant.

\begin{theorem}\label{teoSL}
Let $\tilde L$ as in \eqref{defL2} and let $\tilde r\in L^{\infty}(I)$.
\begin{itemize}
\item[$(i)$] If $1$ is not an eigenvalue of $\tilde L_{\aut}$ in the sense of Definition \ref{def-eigen}, then $\tilde L$ is surjective.
\item[$(ii)$] If $\|\tilde r^+\|_{\infty}<c_p^{-2}$, then $\tilde L$ is surjective.
\end{itemize}
\end{theorem}

\begin{proof}
The proof of $(i)$ is an immediate consequence of Proposition \ref{surSL}, since $\tilde L_{1}=\tilde L$. In order to prove $(ii)$, we choose $\delta< c_p^{-2}- \|\tilde r^+\|_{\infty}$ in the definition \eqref{defL2bis} of $\tilde L_{\aut}$. Let $\varphi_1$ be an eigenfunction of $\aut_1$, normalized so that $\int_0^1 (\tilde r^++\delta)\varphi_1^2=1$. Then
\begin{eqnarray*}
\aut_1  =  \frac{\displaystyle \int_0^1 ((\varphi_1')^2 +(\tilde r^-+\delta)(\varphi_1^2))}{\displaystyle \int_0^1 (\tilde r^++\delta)\varphi_1^2}
\geq \frac{\displaystyle \int_0^1 (\varphi_1')^2}{\displaystyle(\|\tilde r^+\|_{\infty}+\delta)\int_0^1 \varphi_1^2} \geq \frac{c_p^{-2}}{(\|\tilde r^+\|_{\infty}+\delta)}>1,
\end{eqnarray*}
hence $1$ is not an eigenvalue.
\end{proof}
}

\section*{Acknowledgments}
\ch{We are indebted to the referees for their deep and inspiring comments.}
GT received financial support from the MIUR-Italy grant ``Excellence Departments''. LG and GT received financial support from the  MIUR-Italy grant ``Mathematics of active materials: From mechanobiology to smart devices'' (PRIN 2017KL4EF3).

\end{document}